\documentclass[a4paper]{article}
\usepackage{amssymb,amssymb,amsthm,lmodern,url,amsmath,amsbsy,enumerate,tikz}
\usepackage{hyperref}

\usepackage{color}
\usepackage{enumitem}
\usepackage{comment}

\usepackage{blindtext}

\newtheorem{thm}{Theorem}[section]
\newtheorem{defi}[thm]{Definition}
\newtheorem{lem}[thm]{Lemma}
\newtheorem{prop}[thm]{Proposition}
\newtheorem{fact}[thm]{Fact}
\newtheorem{rem}[thm]{Remark}
\newtheorem{ex}[thm]{Example}
\newtheorem{cor}[thm]{Corollary}

\newtheorem{nota}[thm]{Notation}

\newtheoremstyle{mystyle}
    {3pt}
    {3pt}
    {\itshape}
    {}
    {\bfseries}
    {.}
    { }
    {\thmname{#1}\thmnumber{ #2}\thmnote{ (#3)}}
    
\theoremstyle{mystyle}
\newtheorem{ass}{Assumption}[]

\makeatletter
	
	\@addtoreset{equation}{section}

	\@addtoreset{figure}{section}

	\@addtoreset{table}{section}
\makeatother
\title{Counting orbits of certain infinitely generated non-sharp discontinuous groups 
for the anti-de Sitter space}
\author{Kazuki Kannaka\thanks{RIKEN Interdisciplinary Theoretical and    Mathematical Sciences (iTHEMS), 
Wako, Saitama 351-0198, Japan, E-mail adress: kazuki.kannaka@riken.jp}}
\date{}
\begin{document}
\maketitle
\begin{abstract}
Inspired by an example of Gu\'{e}ritaud-Kassel [Geom.\ Topol.\ 2017], 
we construct a family of infinitely generated discontinuous groups $\Gamma$ 
for the $3$-dimensional anti-de Sitter space $\mathrm{AdS}^{3}$. 
These groups are 
\textit{not necessarily sharp} (a kind of ``strong'' proper discontinuity condition
introduced by Kassel and Kobayashi [Adv.\ Math.\ 2016]),
and we give its criterion.
Moreover, we find upper and lower bounds
of the counting $N_{\Gamma}(R)$ of a $\Gamma$-orbit contained 
in a pseudo-ball $B(R)$ as the radius $R$ tends to infinity.
We then find a non-sharp discontinuous group $\Gamma$ 
for which there exist infinitely many $L^2$-eigenvalues 
of the Laplacian on the noncompact anti-de Sitter manifold 
$\Gamma\backslash\mathrm{AdS}^{3}$,
by applying the method established by Kassel-Kobayashi.
We also prove that 
for any increasing function $f$,
there exists a discontinuous group $\Gamma$ for $\mathrm{AdS}^{3}$
such that the counting $N_{\Gamma}(R)$ of 
a $\Gamma$-orbit is larger than $f(R)$
for a sufficiently large $R$.
\end{abstract}

\tableofcontents

\section{Introduction}
\subsection{Construction of \texorpdfstring{$\Gamma_{\nu}(a_{-},a_{+},r,R)$}{} 
and the counting}
\label{intro-count}
In this paper, we construct a family of discrete groups $\Gamma$ of
isometries of the $3$-dimensional anti-de Sitter space $\mathrm{AdS}^{3}$
such that 
\begin{itemize}
\item
$\Gamma$ act properly discontinuously on $\mathrm{AdS}^{3}$;
\item
the counting has an arbitrary growth rate at infinity,
\end{itemize}
generalizing an example of Gu\'{e}ritaud-Kassel \cite[Sect.\ 10.1]{GuKa17}.
By counting, we mean the number of points in a $\Gamma$-orbit contained in a compact
set called a pseudo-ball $B(R)$ of radius $R>0$.

In contrast to the Riemannian case, 
a discrete group of isometries of a pseudo-Riemannian manifold
such as $\mathrm{AdS}^{3}$
may act with non-closed orbits.
We recall some basic notions and facts.
A \textit{pseudo-Riemannian manifold} is a smooth manifold $X$
equipped with a smooth non-degenerate symmetric bilinear form of signature $(p,q)$.
It is \textit{Riemannian} if $q=0$ and \textit{Lorentzian} if $q=1$.
A discrete group $\Gamma$ of isometries of a pseudo-Riemannian manifold $X$
is called a \textit{discontinuous group for $X$} if $\Gamma$ acts on $X$ 
properly discontinuously and freely 
(we include freeness in the definition as in Kobayashi \cite[Def.\ 1.3]{Ko01}).
Then there are at most finite elements in any orbit of a discontinuous group $\Gamma$
contained in any compact subset of $X$, hence we can count them. 
A semisimple symmetric space $X=G/H$ is a typical example of a pseudo-Riemannian 
manifold, of which the isometry group is ``large''.
Kassel and Kobayashi proved in \cite{KaKob16} 
for a discontinuous group $\Gamma(\subset G)$ for an arbitrary 
semisimple symmetric space $G/H$
that the counting is 
at most of exponential growth if $\Gamma$ is \textit{sharp} (a notion for 
``strong'' proper discontinuity, see \cite[Def.\ 4.2]{KaKob16}).

The $3$-dimensional anti-de Sitter space $\mathrm{AdS}^{3}$ 
is the simplest example of a Lorentzian semisimple symmetric space 
that admits infinite discontinuous groups.
Let us recall the counting result of Kassel-Kobayashi \cite{KaKob16}
in this specific setting where $X=\mathrm{AdS}^{3}$
and $G=\mathrm{SL}(2,\mathbb{R})\times \mathrm{SL}(2,\mathbb{R})$.
They considered a compact subset 
$B(R)$ of $X$ called a \textit{pseudo-ball} of radius $R>0$,
of which the volume is of exponential growth as $R\to\infty$,
see Section \ref{psphere}.
They proved that if a discontinuous group $\Gamma\subset G$ is sharp, then the \textit{counting}
\[
N_{\Gamma}(x,R):=\#(\Gamma x\cap B(R)) \ \text{ for } x\in X \text{ and } R>0
\]
has at most an exponential growth uniformly on $x\in X$ (\cite[Lem.\ 4.6 (4)]{KaKob16}):
\begin{align}
\label{count}
\exists A>0,\ \exists a>0,\ \forall x\in X,\ \forall R>0,\ N_{\Gamma}(x,R) \leq Ae^{aR}.
\end{align}
In particular, one has
\begin{align}
\label{weak-count}
\exists a>0,\ \forall x\in X,\ \limsup_{R\to\infty}\frac{N_{\Gamma}(x,R)}{e^{aR}}<\infty.
\end{align}
Any finitely generated discontinuous group for $\mathrm{AdS}^{3}$
is sharp by the results of Kassel \cite[Thm.\ 0.2.13]{Ka09} and Gu\'{e}ritaud-Kassel \cite[Thm.\ 1.8]{GuKa17},
hence its counting always satisfies the exponential growth condition (\ref{count}).

On the other hand, the counting for a \textit{non-sharp} discontinuous
group has not been well-understood.
In this paper, we investigate what can happen about the asymptotic behavior for the
counting $N_{\Gamma}(x,R)$ when $\Gamma$ is non-sharp.
For this, we construct a family of subgroups 
$\Gamma_{\nu}\equiv\Gamma_{\nu}(a_{-},a_{+},r,R)$
of $\mathrm{SL}(2,\mathbb{R})\times \mathrm{SL}(2,\mathbb{R})$
for sufficiently large $\nu\in\mathbb{N}$ 
associated to quadruples $(a_{-},a_{+},r,R)$ of real-valued sequences 
in Section \ref{Gam},
and study how the properties of $\Gamma_{\nu}$ depend on the data $(a_{-},a_{+},r,R)$.
For instance, we find a necessary and sufficient condition for the quadruple $(a_{-},a_{+},r,R)$
that $\Gamma_{\nu}$ is a discontinuous group 
for $\mathrm{AdS}^{3}$ in Proposition \ref{Gamdisc}.
Moreover we determine when 
the $\Gamma_{\nu}$-action on $\mathrm{AdS}^{3}$ is sharp 
in Proposition \ref{non-sharp}.
With these criteria, we present various non-sharp discontinuous groups 
for which different  phenomena happen about
the counting by choosing appropriate data $(a_{-},a_{+},r,R)$:
\begin{thm}
\label{countingabove-intro}
There exists a non-sharp discontinuous group $\Gamma$ for $\mathrm{AdS}^{3}$
such that for any $x\in\mathrm{AdS}^{3}$ and any $R>0$,
the counting $N_{\Gamma}(x,R)$ has at most an exponential growth.
\end{thm}

\begin{thm}
\label{countingbelow}
Let $x\in\mathrm{AdS}^{3}$.
For any increasing function $f \colon \mathbb{R} \to \mathbb{R}_{>0}$, 
there exists a discontinuous group
$\Gamma\equiv\Gamma_{f,x}$ for $\mathrm{AdS}^{3}$ satisfying 
\[\lim_{R\to\infty} \frac{N_{\Gamma}(x,R)}{f(R)}=\infty.\]
\end{thm}

\begin{rem}
\label{rem-riemann}
Theorem \ref{countingbelow} applied to the function $f(R)=\exp(\exp(R))$
shows the existence of a discontinuous group $\Gamma$ for $\mathrm{AdS}^{3}$ such that
\[
\lim_{R\to\infty} \frac{N_{\Gamma}(x,R)}{\mathrm{vol}(B(R+c))}=\infty
\] 
for any $c>0$ since the volume $\mathrm{vol}(B(R))$ 
is of exponential growth as $R\to\infty$. 
Thus an analogue of the Riemannian case (\ref{riemann}) below
does not hold.
\end{rem}

The above theorems deal with the setting 
where the metric tensor of $X$ is indefinite and
$\Gamma$ is a discontinuous group for $X$.
Let us compare them with some known results 
in the following different settings:
\begin{itemize}
\item
$\Gamma$ is a discontinuous group for $X$, but
$X$ is Riemannian (the metric tensor of $X$ is positive definite);
\item
the metric tensor of $X$ is indefinite, 
but $\Gamma$ is not a discontinuous group for $X$ 
(e.g.\ $\Gamma$ is a lattice of the isometry group of $X$).
\end{itemize}

Suppose that $X$ is a complete Riemannian manifold, and that
$\Gamma$ is a discrete group of isometries of $X$.
We write $B(R)$ for the ball of radius $R$ centered at a fixed point in $X$. Then we have (cf.\ Milnor \cite[Thm.\ 1]{Mi68})
\begin{align}
\label{riemann}
\forall x\in X,\ \exists c>0,\ \limsup_{R\to\infty}\frac{\#(\Gamma x\cap B(R))}
{\mathrm{vol}(B(R+c))}<\infty.
\end{align}
The estimate (\ref{riemann}) does not require that $\Gamma$ is finitely generated, but
the Riemannian assumption is crucial as shown in Remark \ref{rem-riemann}.
The inequality (\ref{riemann}) is probably well known, but for the reader's convenience,
we will give a proof in the appendix.

A semisimple symmetric space $X=G/H$ admits a 
$G$-invariant pseudo-Riemannian structure.
Eskin-McMullen \cite{EsMc93} studied the counting of an orbit 
of a lattice $\Gamma$ of $G$ in $X=G/H$.
Let us apply their result \cite[Thm.\ 1.4]{EsMc93} 
to the specific case where $X=\mathrm{AdS}^{3}$
and $H=\mathrm{diag}(\mathrm{SL}(2,\mathbb{R}))\subset G=\mathrm{SL}(2,\mathbb{R})\times \mathrm{SL}(2,\mathbb{R})$.
Then it tells us that if $\Gamma\cap H$ is a lattice in $H$, then at the base point $o=eH\in X$
\begin{align}
\label{Eskin-Mcmullen}
\lim_{R\to\infty}\frac{N_{\Gamma}(o,R)}{\mathrm{vol}(B(R))}=
\frac{\mathrm{vol}(\Gamma\backslash G)}
{\mathrm{vol}((\Gamma\cap H)\backslash H)}.
\end{align}
In the right-hand side, the Haar measures of $G$ and $H$
(and therefore, the induced measures of $\Gamma\backslash G$
and $(\Gamma\cap H)\backslash H$)
are normalized such that the Fubini theorem for the fibration
$H\rightarrow G\rightarrow X=G/H$ is given by
$dg=dxdh$, where $dx$ is the volume element of the anti-de Sitter space
$X=\mathrm{AdS}^{3}$ (see Section \ref{preliminary}).
We note that their setting is different from ours: 
$\Gamma$ in \cite{EsMc93} is a lattice of $G$,
hence does not act properly discontinuously on $X$.

We summarize these results about 
the asymptotic behaviors 
of $N_{\Gamma}(x,R)$ in each setting in Table \ref{table-asymptotic} below:
\begin{table}[htb]
\label{table-asymptotic}
\begin{center}
\caption{The asymptotic behaviors of $N_{\Gamma}(x,R)$
}
\begin{tabular}{|c|c|c|c|}
\hline
& $\Gamma$ & $x\in\mathrm{AdS}^{3}$ & $N_{\Gamma}(x,R)$ \\
\hline
Eskin-McMullen \cite{EsMc93} 
& $\forall$ lattice in $G$ & special $x$ & $\sim Ae^{R}$ \\ 
\hline
Kassel-Kobayashi \cite{KaKob16} & $\forall$ sharp discont.\ gp.\ 
& general $x$
& $\leq Ae^{aR}$\\
\hline
Theorem \ref{countingabove-intro} & $\exists$ non-sharp discont.\ gp.\
& general $x$ & $\leq Ae^{aR}$\\
\hline
Theorem \ref{countingbelow} & $\exists$ non-sharp discont.\ gp.\
& general $x$ & $\gg \exp(e^{R})$ \\
\hline
\end{tabular}
\end{center}
\end{table}

\begin{rem}
Kassel-Kobayashi \cite{KaKob16} gave a 
uniform estimate of $N_{\Gamma}(x,R)$ with respect to $x\in\mathrm{AdS}^{3}$. 
We prove such a uniform estimate for Theorem \ref{countingabove-intro},
but not for Theorem \ref{countingbelow}.
\end{rem}

\subsection{Spectrum of the Laplacian on
\texorpdfstring{$\Gamma_{\nu}\backslash\mathrm{AdS}^{3}$}{}}
Let $\Gamma$ be a discontinuous group for
the anti-de Sitter space $X:=\mathrm{AdS}^{3}$.
Then the quotient space 
$X_{\Gamma}:=\Gamma\backslash X$ is a $C^{\infty}$-manifold and
the quotient map $X\rightarrow X_{\Gamma}$
is a smooth covering.
Thus $X_{\Gamma}$ inherits an anti-de Sitter structure from $X$,
and in particular, is a Lorentzian manifold.
As in the Riemannian case, 
one defines the Laplacian 
$\square_{X_{\Gamma}}:=\mathrm{div}\circ\mathrm{grad}$,
a second-order differential operator on $X_{\Gamma}$. 

Kassel-Kobayashi \cite{KaKob16} initiated the study of global analysis 
on the anti-de Sitter manifold $X_{\Gamma}$ 
(actually in a much more general setting).
They studied the \textit{discrete spectrum}, namely 
the set of $L^{2}$-eigenvalues of the Laplacian $\square_{X_{\Gamma}}$
on $X_{\Gamma}$, denoted by
\[
\mathrm{Spec}_{d}(\square_{X_{\Gamma}}):= 
\{ \lambda \in \mathbb{C} \mid \exists f \in L^2(X_{\Gamma})\smallsetminus\{0\},\ 
\square_{X_{\Gamma}} f = \lambda f \text{ in the weak sense}\}.
\]
Here $L^2(X_{\Gamma})$ is the Hilbert space of square integrable functions 
on $X_{\Gamma}$ with respect to the Radon measure induced by the Lorentzian structure.
We note that in contrast to the Riemannian case where the Laplacian
is an elliptic differential operator, the Laplacian for the Lorentzian manifold $X_{\Gamma}$
is a hyperbolic operator, and thus eigenfunctions may and may not be smooth functions
by the failure of the elliptic regularity theorem (see \cite[Sect.\ 3.1]{KaKob19} for example).
Kassel-Kobayashi \cite{KaKob16} proved the following: 
if $\Gamma$ is sharp, 
then there exists $m_{0}=m_{0}(\Gamma)>0$ such that
\[
\mathrm{Spec}_{d}(\square_{X_{\Gamma}})\supset \{4m(m-1)\mid m\in\mathbb{Z}
\text{ and } m> m_{0}\}.
\]
In particular, they proved that the discrete spectrum $\mathrm{Spec}_{d}(\square_{X_{\Gamma}})$ is infinite 
in the setting where $\Gamma$ is sharp.

A natural question would be whether the Laplacian $\square_{X_{\Gamma}}$ 
still has an $L^{2}$-eigenvalue if the discontinuous group $\Gamma$ is non-sharp. 
As an application of the sharpness criterion (Proposition \ref{non-sharp})
and an upper estimate of the counting as in Theorem \ref{countingabove-intro}, we see that the machinery developed in \cite{KaKob16} also can be applied to \textit{some} non-sharp discontinuous groups, and prove:

\begin{thm}[see Theorem \ref{countingabove} and Example \ref{ex-countingabove}]
\label{non-sharp-spectrum-intro}
There exist a non-sharp discontinuous group $\Gamma$ for $\mathrm{AdS}^{3}$
and $m_{0}=m_{0}(\Gamma)>0$ such that
\[
\mathrm{Spec}_{d}(\square_{X_{\Gamma}})\supset
\{4m(m-1)\mid m\in\mathbb{Z}\text{ and } m> m_{0}\}.
\]
\end{thm}

\subsection{Organization of the paper}
\label{organization-sec}
In Section \ref{preliminary}, 
we give preliminary results, 
including a pseudo-ball $B(R)$
and the Kobayashi-Benoist properness criterion 
applied to our $\mathrm{AdS}^{3}$ setting.
In Section \ref{Gam}, 
we construct a family of infinitely generated Schottky-like
discontinuous groups $\Gamma_{\nu}\equiv
\Gamma_{\nu}(a_{-},a_{+},r,R)$ for $\mathrm{AdS}^{3}$
associated to quadruples $(a_{-},a_{+},r,R)$ of real-valued sequences
satisfying some conditions for a sufficiently large $\nu\in\mathbb{N}$.
Moreover, we recall the notion of sharpness for discontinuous groups,
and find a necessary and sufficient condition on the quadruple $(a_{-},a_{+},r,R)$ such that 
$\Gamma_{\nu}$ is sharp.
In Section \ref{countingbelowsec}, we find a lower bound for the counting 
$N_{\Gamma_{\nu}}(x,R)$, and prove Theorem \ref{countingbelow}.
In Section \ref{dspec}, we find a sufficient condition on the quadruple
$(a_{-},a_{+},r,R)$ such that the counting $N_{\Gamma_{\nu}}(x,R)$ is at most of
exponential growth, and complete the proof of Theorem \ref{countingabove-intro}
with the sharpness criterion given in Section \ref{Gam}.
The proof of Theorem \ref{non-sharp-spectrum-intro} is then given by
applying the method established by Kassel-Kobayashi \cite{KaKob16}.

\vspace{\baselineskip}
\noindent
\textbf{Notation}. $\mathbb{N}=\{0,1,2,\ldots\}$ and $\mathbb{N}_{+}=\{1,2,3,\ldots\}$.

\section{Preliminary results about \texorpdfstring{$\mathrm{AdS}^{3}$}{}}
\label{preliminary}
In this section, we collect some preliminary results about $\mathrm{AdS}^{3}$
that will be needed for later sections.

Let $V$ be a four-dimensional real vector space equipped with
a quadratic form $Q$ of signature $(2,2)$ on $V$,
and $X$ the hypersurface given by $X=\{v\in V\mid Q(v)=1\}$.
The tangent space $T_{v}X$ at $v\in X$ is identified with 
the orthogonal complement $(\mathbb{R}v)^{\bot}$ in $V$ with respect to $Q$.
The restriction of $-Q$ to the hyperplane 
$(\mathbb{R}v)^{\bot}$ is a quadratic form of signature $(2,1)$,
which induces a Lorentzian structure on $X$ 
with constant sectional curvature $-1$.
The resulting Lorentzian manifold is called 
the $3$-dimensional anti-de Sitter space $\mathrm{AdS}^{3}$.

\subsection{Pseudo-balls in \texorpdfstring{$\mathrm{AdS}^{3}$}{}}
\label{psphere}

In this subsection, 
we consider pseudo-balls $B(R)$ on the Lorentzian manifold 
$\mathrm{AdS}^{3}$. 
We work with coordinates on $\mathrm{AdS}^{3}$ 
by choosing $V=\mathrm{M}(2,\mathbb{R})$ and $Q=\mathrm{det}$.
Then $\mathrm{AdS}^{3}$ is identified with $\mathrm{SL}(2,\mathbb{R})$.
The direct product group 
$\mathrm{SL}(2,\mathbb{R})\times \mathrm{SL}(2,\mathbb{R})$
acts on $V=\mathrm{M}(2,\mathbb{R})$
by left and right multiplication, which induces an isometric and transitive action
on $\mathrm{AdS}^{3}$. Thus
\[
\mathrm{AdS}^{3}\cong (\mathrm{SL}(2,\mathbb{R})\times\mathrm{SL}(2,\mathbb{R}))
/\mathrm{diag}(\mathrm{SL}(2,\mathbb{R})).
\]

Let $o$ be the base point in $\mathrm{AdS}^{3}$ corresponding to the 
identity matrix in $\mathrm{SL}(2,\mathbb{R})$.
The \textit{pseudo-distance} $\|g\|(\geq0)$ of $g\in\mathrm{SL}(2,\mathbb{R})$ 
from the base point $o$ is defined by the formula
\begin{align}
\label{def-pdist}
2\cosh\|g\|=\operatorname{Tr}(^tgg).
\end{align}
We give two alternative definitions of the pseudo-distance $\|g\|$ as below.

First, for $\theta\in[0,2\pi]$ and $t\geq 0$,
we set
$k(\theta):=
\begin{pmatrix}
\cos \theta & -\sin\theta \\
\sin\theta & \cos \theta
\end{pmatrix}$ and
$a(t) :=\begin{pmatrix}
e^{t} & 0 \\
0 & e^{-t}
\end{pmatrix}$.
Any element $g\in\mathrm{SL}(2,\mathbb{R})$
can be expressed by the Cartan decomposition $g=k(\theta_{1})a(t)k(\theta_{2})$ with 
$\theta_{1},\theta_{2}\in[0,2\pi]$ and unique $t\geq0$.
Then (\ref{def-pdist}) implies 
\begin{align}
\label{pdist-cartan}
\|g\|=2t.
\end{align}
This interpretation shows readily that 
the map $\|\cdot\|\colon\mathrm{SL}(2,\mathbb{R})\rightarrow[0,\infty)$
is proper and that for any $R>0$,
\begin{align}
\label{pseudo-ball}
B(R):=\{g\in\mathrm{SL}(2,\mathbb{R})\mid \|g\|\leq R\}
\end{align}
is a compact subset of $\mathrm{SL}(2,\mathbb{R})$, to which we refer as
the \textit{pseudo-ball} of radius $R$.
The family $\{B(R)\}_{R>0}$ is well-rounded 
(Eskin-McMullen \cite[Thm.\ 6.1]{EsMc93}).

Second, we realize the hyperbolic space $\mathbb{H}^{2}$ as the upper-half plane 
$\{x+\sqrt{-1} y\in\mathbb{C}\mid y>0\}$ endowed with the metric
tensor $ds^{2}=y^{-2}(dx^{2}+dy^{2})$.
We write $d_{\mathbb{H}^{2}}$ for the hyperbolic distance of $\mathbb{H}^{2}$.
The group $\mathrm{SL}(2,\mathbb{R})$ acts isometrically on $\mathbb{H}^{2}$ 
by linear fractional transformations.
In this model, the pseudo-distance $\|g\|$ is computed by (\ref{pdist-cartan})
as follows:
\begin{lem}[see e.g.\ {\cite[(A.1) and (A.2)]{GuKa17}}]
\label{hypdist}
For any $g \in \mathrm{SL}(2,\mathbb{R})$,
\[\|g\|=d_{\mathbb{H}^{2}}(g\sqrt{-1},\sqrt{-1}).\]
In particular, for any point $x + \sqrt{-1}y \in \mathbb{H}^{2}$,
\[2 \cosh d_{\mathbb{H}^{2}}(x + \sqrt{-1}y, \sqrt{-1}) = \frac{x^2 + y^2 + 1}{y}.\]
\end{lem}

The following properties of the pseudo-distance follow from Lemma \ref{hypdist}:
\begin{lem}
\label{prop-pdist}
For $g,g'\in\mathrm{SL}(2,\mathbb{R})$,
\begin{enumerate}[label=$(\arabic*)$]
\item
$\|g^{-1}\|=\|g\|$.
\item
$|\|g\|-\|g'\||\leq \|gg'\|\leq\|g\| +\|g'\|$.
\end{enumerate}
\end{lem}
The Jacobian of the Cartan decomposition
$(0,2\pi)\times (0,\infty) \times (0,2\pi)
\rightarrow \mathrm{SL}(2,\mathbb{R})$ defined by 
$(\theta_{1},t,\theta_{2})\mapsto k(\theta_{1})a(t)k(\theta_{2})$
equals $\sinh (2t)$ with respect to this Lorentzian structure on 
$\mathrm{SL}(2,\mathbb{R})\cong\mathrm{AdS}^{3}$ 
and the standard metrics of the intervals $(0,2\pi)$ and $(0,\infty)$.
Hence the following integral formula holds:
\begin{align}
\label{cartan-integral}
\int_{\mathrm{SL}(2,\mathbb{R})}f(g)dg =
\int_{0}^{2\pi}\int_{0}^{\infty}\int_{0}^{2\pi}
f(k(\theta_{1})a(t)k(\theta_{2}))\sinh(2t)d\theta_{1}dtd\theta_{2}.
\end{align}
Therefore, the volume $\mathrm{vol}(B(R))$ equals $2\pi^{2}(\cosh(R) -1)$
since $k(\theta_{1})a(t)k(\theta_{2})\in B(R)$ if and only if $2t \leq R$.

\subsection{Discontinuous groups for \texorpdfstring{$\mathrm{AdS}^{3}$}{}}
\label{criterion}

Let $G$ be a Lie group,
$H$ a closed subgroup of $G$,
and $\Gamma$ a discrete subgroup of $G$,
which acts naturally on $X:=G/H$ from the left.
In this subsection, 
we explain the Kobayashi-Benoist criterion for the proper discontinuity
of the $\Gamma$-action on $X$
applied to our specific setting where 
$G=\mathrm{SL}(2,\mathbb{R})\times\mathrm{SL}(2,\mathbb{R})$,
$H=\mathrm{diag}(\mathrm{SL}(2,\mathbb{R}))$, and 
$X=\mathrm{AdS}^{3}$.

Throughout this paper, we mean by a discontinuous group for $X$
a discrete subgroup $\Gamma$ of $G$ acting 
properly discontinuously and freely on $X$ (Kobayashi \cite[Def.\ 1.3]{Ko01}).
A torsion-free discrete subgroup $\Gamma$ of $G$ is a discontinuous group for $X$ if and only if 
$\Gamma$ acts properly discontinuously on $X$.
Proper discontinuity is a serious constraint
when the isotropy subgroup of $G$
on $X$ is noncompact. Geometrically, one should note that 
not every discrete subgroup of isometries can act 
properly discontinuously on a pseudo-Riemannian manifold $X$.
Kobayashi \cite{Kob96} and Benoist \cite{Ben96}
established a properness criterion for reductive $G$
generalizing the original properness criterion of Kobayashi \cite{Kob89}.

Applying the Kobayashi-Benoist properness criterion to our specific setting,
we can determine whether the $\Gamma$-action on $\mathrm{AdS}^{3}$
is properly discontinuous in terms of the pseudo-distance defined in Section \ref{psphere}
as follows: 

\begin{fact}[Kobayashi {\cite[Thm.\ 3.4]{Kob96}} and Benoist {\cite[Thm.\ 5.2]{Ben96}}]
\label{KoBe}

Let $\Gamma$ be a discrete subgroup of 
$\mathrm{SL}(2,\mathbb{R})\times\mathrm{SL}(2,\mathbb{R})$.
The following are equivalent:

\begin{enumerate}
\item[$(\mathrm{i})$]
The action of $\Gamma$ on $\mathrm{AdS}^{3}$ is properly discontinuous.
\item[$(\mathrm{ii})$]
For any $C > 0$, the set $\{ (\alpha,\beta) \in \Gamma \mid 
\left | \|\alpha\| - \|\beta\| \right| <C \}$ is finite. 
\end{enumerate}
\end{fact}

\section{Discontinuous groups 
\texorpdfstring{$\Gamma_{\nu}(a_{-},a_{+},r,R)$}{} for \texorpdfstring{$\mathrm{AdS}^{3}$}{}}
\label{Gam}
In this section, 
we introduce a family of Schottky-like subgroups
$\Gamma_{\nu}(a_{-},a_{+},r,R)$ of 
$G=\mathrm{SL}(2,\mathbb{R})\times\mathrm{SL}(2,\mathbb{R})$
in Definition \ref{def-gam} associated to the following data:
\begin{itemize}
\item
$\nu\in\mathbb{N}$;
\item
$a_{-},a_{+}, r,R\colon\mathbb{N}\rightarrow\mathbb{R}_{>0}$
satisfying Assumptions 1--3 below.
\end{itemize}
We give a properness criterion and a sharpness criterion for 
the action of $\Gamma_{\nu}(a_{-},a_{+},r,R)$ on $\mathrm{AdS}^{3}$ for any sufficiently large $\nu\in\mathbb{N}$ in terms of the quadruple $(a_{-},a_{+},r,R)$.
Furthermore, we also give sufficient conditions on $\nu$
for discreteness, properness, 
and sharpness of $\Gamma_{\nu}(a_{-},a_{+},r,R)$.
In Section \ref{subsection:construction}, 
we construct an infinitely generated, free discrete subgroup $\Gamma_{\nu}(a_{-},a_{+},r,R)$ of $G$ (Proposition \ref{prop:discrete-free}).
In Section \ref{subsection:key}, we introduce 
a constant $\varepsilon(\nu)$ and prove a key proposition (Proposition \ref{prop:key-ineq}) which will be repeated in later sections.
In Section \ref{subsection:proper}, based on the Kobayashi-Benoist properness criterion (Fact \ref{KoBe}), we give a criterion for the $\Gamma_{\nu}(a_{-},a_{+},r,R)$-action on $\mathrm{AdS}^{3}$ to be proper (Proposition \ref{proper}).
In Section \ref{sharpsec}, we give a criterion for the $\Gamma_{\nu}(a_{-},a_{+},r,R)$-action on $\mathrm{AdS}^{3}$ to be $(c,0)$-sharp 
in the sense of Kassel-Kobayashi \cite{KaKob16} (Proposition \ref{non-sharp}).

\subsection{Construction of discrete subgroups 
\texorpdfstring{$\Gamma_{\nu}(a_{-},a_{+},r,R)$}{}}

\label{subsection:construction}

We introduce a coordinate map 
$\tau\colon\mathbb{R}\times\mathbb{R}\times\mathbb{R}_{>0}
\rightarrow \mathrm{SL}(2,\mathbb{R})$ by 
\begin{align}
\label{tau-def}
\tau=\tau(x_{-},x_{+},u):=\frac{1}{u}
\begin{pmatrix}
x_{+} & -(x_{-}x_{+}+u^{2}) \\
1 & -x_{-}
\end{pmatrix}
\in\mathrm{SL}(2,\mathbb{R}).
\end{align}

\begin{defi}
\label{def-gam}
Let $(a_{-},a_{+},r,R)$ be a quadruple of positive real valued sequences. 
Then we define a sequence of elements 
$(\alpha_{k},\beta_{k}) \in G$ by 
\begin{align}
\label{alpha-beta-def}
\alpha_k := \tau(a_{-}(k),a_{+}(k),r(k)),\ 
\beta_k := \tau(a_{-}(k),a_{+}(k),R(k)) \in \mathrm{SL}(2,\mathbb{R}).
\end{align}
For $\nu\in\mathbb{N}$, we define $\Gamma_{\nu}(a_{-},a_{+},r,R)$ as the subgroup 
of $G$ generated by $\{(\alpha_k,\beta_k)\mid k=\nu,\nu+1,\ldots \}$.
\end{defi}

\begin{nota}
\label{notation-j-rho-F}
Let $F^{\infty}$ denote the free group generated by countably many elements 
$\{\gamma_k\}_{k \in \mathbb{N}}$.
Let $(\alpha_{k},\beta_{k})\in G$ be a sequence of elements associated to
a quadruple $(a_{-},a_{+},r,R)$ by (\ref{alpha-beta-def}).
Then we write $j\colon F^{\infty}\rightarrow \mathrm{SL}(2,\mathbb{R})$ and
$\rho\colon F^{\infty}\rightarrow \mathrm{SL}(2,\mathbb{R})$
for the group homomorphisms
such that $j(\gamma_k) = \alpha_{k}$ and $\rho(\gamma_k) = \beta_{k}$
for all $k \in \mathbb{N}$.
For $\nu\in\mathbb{N}$, let $F_{\nu}^{\infty}$
be the subgroup of $F^{\infty}$ generated by $\{\gamma_k\}_{k=\nu}^{\infty}$.
\end{nota}
Then, by Definition \ref{def-gam}, 
\begin{align*}
\Gamma_{\nu}(a_{-},a_{+},r,R) = 
\{(j(\gamma), \rho(\gamma)) \mid \gamma \in F^{\infty}_{\nu}\}.
\end{align*}

\begin{ex}
The subgroup $\Gamma_{\nu}(a_{-},a_{+},r,R)$ for
$(a_{-}(k),a_{+}(k),r(k),R(k)) = (k^2,k^2 + k,1,\log k)$
coincides with $\Gamma_{\nu}^{j,\rho}$ in Gu\'{e}ritaud-Kassel 
\cite[Sect.\ 10.1]{GuKa17}.
\end{ex}

\begin{prop}
\label{prop:discrete-free}
Suppose that a quadruple 
of positive real valued sequences $(a_{-},a_{+},r,R)$ 
satisfies the following assumptions:
\begin{ass}
    \label{ass:assume1}
    For any sufficiently large integer $k$, we have 
    \begin{align}
    r(k) &< R(k), \nonumber\\
    a_{-}(k) + R(k) &< a_{+}(k) - R(k), \label{ineq:assume1}\\
    a_{+}(k) + R(k) &< a_{-}(k+1) - R(k+1). \nonumber
    \end{align}
\end{ass} 
\begin{ass}
    \label{ass:assume2}
    $\displaystyle \lim_{k \to \infty}a_{+}(k) = \lim_{k\to\infty}a_{-}(k)=\infty.$
\end{ass}
Let $\nu\in\mathbb{N}$. 
If \eqref{ineq:assume1} holds 
for any integer $k\geq \nu$, then the subgroup $\Gamma_{\nu}(a_{-},a_{+},r,R)$ of $G$ is discrete and free.
\end{prop}

The proof is based on the ping-pong lemma.
For this, we need some setups. Let $|\cdot|$ denote the Euclidean norm in the upper-half
plane $\mathbb{H}^{2} \subset \mathbb{C}$. Associated to the quadruple $(a_{-},a_{+},r,R)$, we set
\begin{align}
\label{half-disks}
A_{k}^{\epsilon}:=\{z\in \mathbb{H}^{2} \mid |z-a_{\epsilon}(k)|\leq r(k)\},\ 
B_{k}^{\epsilon}:=\{z\in \mathbb{H}^{2} \mid |z-a_{\epsilon}(k)|\leq R(k)\}.
\end{align}
for $k\in\mathbb{N}$ and $\epsilon\in\{+,-\}$, see Figure \ref{arrangement}.
Then we claim:
\begin{itemize}
\item
$A_{\nu}^{-},A_{\nu}^{+},A_{\nu+1}^{-},A_{\nu+1}^{+},\ldots$ 
are disjoint;
\item
$\alpha_{k}(\mathbb{H}^{2}\smallsetminus A_{k}^{-})\subset A_{k}^{+}$ 
for $k\geq\nu$;
\item
$\bigcup_{k\geq\nu}(A_{k}^{-}\cup A_{k}^{+})$ 
is a proper closed subset of $\mathbb{H}^{2}$.
\end{itemize}
The first claim is immediate from the fact that the inequalities \eqref{ineq:assume1} in Assumption \ref{ass:assume1} hold for any integer $k\geq \nu$.

The second claim is implied by the following key property of 
the map $\tau=\tau(x_{-},x_{+},u)$ in (\ref{tau-def}):
\begin{align}
\label{tau-property-arrange}
|z-x_{-}|> u \text{ if and only if } |\tau(z)-x_{+}|< u \text{ for }z\in\mathbb{H}^{2},
\end{align}
which is readily seen from the identity
\begin{align}
\label{tau-property}
\tau(z)-x_{+}=-u^{2}(z-x_{-})^{-1}.
\end{align}

To prove the third claim, it suffices to show
\begin{align}
\label{eq:third-claim}
\lim_{k\to\infty}(|z-a_{\epsilon}(k)| - r(k)) = \infty
\end{align}
for any $z\in \mathbb{H}^{2}$ and 
any $\epsilon\in\{+,-\}$.
Recall $a_{\epsilon}(k)>0$ for any $k\in\mathbb{N}$. 
Hence, by Assumption 1, we have 
\[
|z-a_{\epsilon}(k)| - r(k) \geq a_{\epsilon}(k)-|z|-r(k)
\geq a_{+}(k-1)-|z|.
\]
By Assumption \ref{ass:assume2}, we obtain \eqref{eq:third-claim}.
This proves the third claim.

We are ready to prove Proposition \ref{prop:discrete-free}.
\begin{proof}[Proof of Proposition \ref{prop:discrete-free}]
The subgroup of $\mathrm{SL}(2,\mathbb{R})$ generated by 
$\{\alpha_{k}\mid k=\nu,\nu+1,\ldots\}$ is free and discrete  
by the standard ping-pong argument,
namely, by applying Lemma \ref{ping-pong} below to 
$H=\mathrm{SL}(2,\mathbb{R})$, $Y=\mathbb{H}^{2}$, and 
$Y_{k}^{\pm}=A_{k+\nu}^{\pm}$.
Hence $\Gamma_{\nu}(a_{-},a_{+},r,R)$ is also free and discrete.
\end{proof}

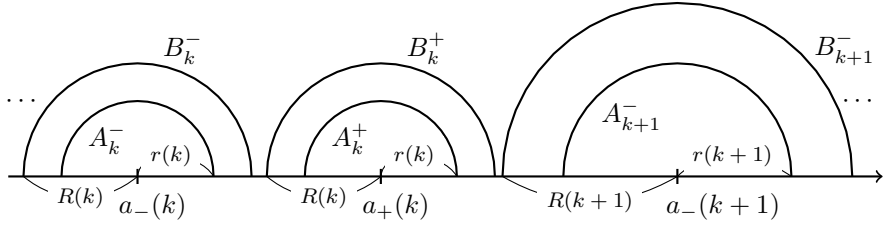
\begin{figure}
\begin{tikzpicture}[scale=1.0]
\draw [thick,->] (-0.2,0.0) -- (11.3,0.0);

\node at (0,1) {$\cdots$};

\draw [thick,-] (1.5,-0.1) -- (1.5,0.1);
\draw [thick,domain=0:180] plot ({1.5+1.5*cos(\x)}, {1.5*sin(\x)});
\draw [thick,domain=0:180] plot ({1.5+cos(\x)}, {sin(\x)});
\node at (1.7,-0.4) {$a_{-}(k)$};
\node at (1.1,0.5){$A_{k}^{-}$};
\node at (2.1, 1.7){$B_{k}^{-}$};
\draw (0,0) to 
[out=-40, in=220, edge node={node [midway,fill=white]{\footnotesize $R(k)$}}] (1.5,0);
\draw (1.5,0) to 
[out=85, in=120, edge node={node [midway,fill=white]{\footnotesize $r(k)$}}] (2.5,0);

\draw [thick,-] (4.7,-0.1) -- (4.7,0.1);
\draw [thick,domain=0:180] plot ({4.7+1.5*cos(\x)}, {1.5*sin(\x)});
\draw [thick,domain=0:180] plot ({4.7+cos(\x)}, {sin(\x)});
\node at (4.9,-0.4) {$a_{+}(k)$};
\node at (5.2, 0.3) {$r(k)$};
\node at (3.95, -0.3) {$R(k)$};
\node at (4.3,0.5){$A_{k}^{+}$};
\node at (5.3, 1.7){$B_{k}^{+}$};
\draw (3.2,0) to 
[out=-40, in=220, edge node={node [midway,fill=white]{\footnotesize $R(k)$}}] (4.7,0);
\draw (4.7,0) to 
[out=85, in=120, edge node={node [midway,fill=white]{\footnotesize $r(k)$}}] (5.7,0);

\draw [thick,-] (8.6,-0.1) -- (8.6,0.1);
\draw [thick,domain=0:180] plot ({8.6+2.3*cos(\x)}, {2.3*sin(\x)});
\draw [thick,domain=0:180] plot ({8.6+1.5*cos(\x)}, {1.5*sin(\x)});
\node at (9.2,-0.4) {$a_{-}(k+1)$};
\node at (8,0.8){$A_{k+1}^{-}$};
\node at (10.8,1.7){$B_{k+1}^{-}$};
\draw (6.3,0) to 
[out=-30, in=210, edge node={node [midway,fill=white]{\footnotesize $R(k+1)$}}] (8.6,0);
\draw (8.6,0) to 
[out=45, in=145, edge node={node [midway,fill=white]{\footnotesize $r(k+1)$}}] (10.1,0);

\node at (11,1) {$\cdots$};
\end{tikzpicture}
\caption{$A_{k}^{\pm}$ and $B_{k}^{\pm}$ in $\mathbb{H}^{2}$} \label{arrangement}
\end{figure}

\begin{lem}[The ping-pong lemma]
\label{ping-pong}
Let $H$ be a topological group acting continuously on a topological space $Y$,
and $\Gamma$ the subgroup generated by $h_{0}, h_{1},\ldots\in H$. 
Suppose that there exist disjoint closed subsets 
$Y_{0}^{-},Y_{0}^{+},Y_{1}^{-},Y_{1}^{+},\ldots$ of $Y$
satisfying the following:
\begin{enumerate}[label=$(\roman*)$]
\item
$h_{k}(Y\smallsetminus Y_{k}^{-})\subset Y_{k}^{+}$ for any $k\in\mathbb{N}$.
\item
$\bigcup_{k\in\mathbb{N}}(Y_{k}^{-}\cup Y_{k}^{+})$ is a proper closed subset of $Y$.
\end{enumerate}
Then, $\Gamma$ is a free discrete subgroup of $H$.
\end{lem}

Although the proof of Lemma \ref{ping-pong} 
is standard,
we give a proof for the sake of completeness.

\begin{proof}
The conditions (i) and (ii) may be restated as 
$h_{k}^{-1}(Y\smallsetminus Y_{k}^{+})\subset Y_{k}^{-}$ 
for any $k\in\mathbb{N}$ and
$U:=Y\smallsetminus \bigcup_{i\in\mathbb{N}}(Y_{i}^{-}\cup Y_{i}^{+})$
is a non-empty open subset of $Y$,
respectively.
Take any $h=h_{i_{0}}^{s_{0}}\cdots h_{i_{n}}^{s_{n}} \in \Gamma$.
Suppose that this is a reduced expression, namely, 
$s_{0},\ldots,s_{n}\in\{1,-1\}$ and $s_{k}=s_{k+1}$ whenever $i_{k}=i_{k+1}$ for $0\leq k<n$.
Then we have $h(U)\subset Y_{i_{0}}^{s_{0}}$ and thus $h(U)\cap U=\emptyset$.
Hence $h\neq e$ and $\{h_{0}, h_{1},\ldots\}$ is a free generator of $\Gamma$.
Take a neighborhood $V$ of $e$ in $H$
and a non-empty open subset $W$ of $U$ such that $V\cdot W\subset U$.
Then $\Gamma\cap V=\{e\}$ and thus $\Gamma$ is discrete in $H$.
\end{proof}

As a byproduct of the above discussion, we also see the following:
\begin{cor}
\label{cor:discrete-free}
Let a quadruple $(a_{-},a_{+},r,R)$ and $\nu\in \mathbb{N}$
be as in the setting of Proposition \ref{prop:discrete-free}.
Let $z\in \mathbb{H}^{2}\smallsetminus \bigcup_{k\geq\nu}(B_{k}^{-}\cup B_{k}^{+})$ and $\gamma\in F^{\infty}_{\nu}\smallsetminus\{e\}$.
If the expression $\gamma=\gamma_{i_{0}}^{s_{0}}\ldots\gamma_{i_{n}}^{s_{n}}$ 
is reduced, 
then we have $j(\gamma)z \in A_{k_{0}}^{s_{0}}$ 
and $\rho(\gamma)z \in B_{k_{0}}^{s_{0}}$.
Here we have used the convention that 
$A_{k}^{\pm1}=A_{k}^{\pm}$ and $B_{k}^{\pm1}=B_{k}^{\pm}$.
\end{cor}

\begin{ex}
\label{exgamma}
The quadruples 
$(a_{-},a_{+},r,R)$ in (1)--(3) of Table \ref{table-quadruple}
satisfy Assumptions \ref{ass:assume1} and \ref{ass:assume2} in Proposition \ref{prop:discrete-free}.
For the reader's convenience, we list in Table \ref{table-quadruple} also the asymptotic behaviors of the counting $N_{\Gamma}(x,R)$ 
as $R$ tends to infinity where $\Gamma$
are the discontinuous groups $\Gamma_{\nu}(a_{-},a_{+},r,R)$ associated to
the quadruples $(a_{-},a_{+},r,R)$ in (1) and (3). We refer to 
Example \ref{ex-countingabove} and Example \ref{expexp} below 
for details about the counting.
\end{ex}
\begin{table}[htb]
\caption{Examples of $(a_{-},a_{+},r,R)$ satisfying Assumptions \ref{ass:assume1} and \ref{ass:assume2}}
\label{table-quadruple}
\begin{center}
\begin{tabular}{|c|c|c|c|c|c|}
\hline
& $a_{-}(k)$ & $a_{+}(k)$ & $r(k)$ & $R(k)$ & $N_{\Gamma}(x,R)$ \\
\hline
(1) & $\exp(e^{k})$ & $\exp(e^{k+\frac{1}{2}})$ & $1$ & $e^k$ & $\leq 2^{2R}$ \\ 
\hline
(2) & $k^2$ & $k^2 + k$ & $1$ & $\log k$ & \\
\hline
(3) & $\log k$ & $\log(k + \frac{1}{2})$ & $(k^2 \log k)^{-1}$ & $k^{-2}$
& $\geq \exp(e^{\frac{R}{4}})-\nu$ \\
\hline 
\end{tabular}
\end{center}
\end{table}

\subsection{A key proposition}
\label{subsection:key}
In this subsection, we prove a key proposition (Proposition \ref{prop:key-ineq}) which will be applied to:
\begin{itemize}
    \item 
    Determine the properness of the $\Gamma_{\nu}(a_{-},a_{+},r,R)$-action on $\mathrm{AdS}^{3}$ (Proposition \ref{proper});
    \item 
    Determine the sharpness of the $\Gamma_{\nu}(a_{-},a_{+},r,R)$-action on $\mathrm{AdS}^{3}$ in the sense of Kassel-Kobayashi \cite{KaKob16}(Proposition \ref{non-sharp});
    \item
    Estimate an upper bound of the counting $N_{\Gamma_{\nu}(a_{-},a_{+},r,R)}(x,R)$ (Theorem \ref{countingabove} \ref{item:countingabove}).
\end{itemize}

For this, in addition to Assumptions \ref{ass:assume1} and \ref{ass:assume2} in Proposition \ref{prop:discrete-free}, we need to impose another condition on a quadruple $(a_{-},a_{+},r,R)$ of positive real valued sequences.
For $\nu\in\mathbb{N}$, we set 
\begin{align}
\label{def-eta}
\eta(\nu):=\sup \left\{\left| \frac{R(k)}{a_{\delta}(k) - a_{\epsilon}(\ell)} \right| 
\middle|\ \delta,\epsilon \in\{+,-\} \text{ and } k,\ell\geq \nu \text{ s.t.\ }
(k,\delta)\neq(\ell,\epsilon)\right\}.
\end{align}
From now on, we always assume the following:
\begin{ass}
    \label{ass:assume3}
    $\displaystyle \lim_{\nu\to\infty} \eta(\nu)=0$.
\end{ass}

Furthermore, we introduce a constant $\varepsilon(\nu)$.
Let $(a_{-},a_{+},r,R)$ be a quadruple of positive real valued sequences satisfying Assumptions \ref{ass:assume1}--\ref{ass:assume3}.
Let $(\alpha_{k},\beta_{k}) \in G = \mathrm{SL}(2,\mathbb{R})\times \mathrm{SL}(2,\mathbb{R})$ be the sequence of the elements
associated to the quadruple $(a_{-},a_{+},r,R)$ by (\ref{alpha-beta-def}).
For $\nu\in\mathbb{N}$, we put
\begin{align}
\label{varepsilon-def}
\varepsilon(\nu):=\max\left\{\max_{k\geq\nu}
\left\{\frac{24R(k)}{a_{-}(k)},\ \frac{6(R(k)^{2}+1)}{(a_{-}(k)-R(k))^{2}}\right\},\
8\eta(\nu)\right\}.
\end{align}
The sequence $\varepsilon= \{\varepsilon(\nu)\}_{\nu\in\mathbb{N}}$
is monotone decreasing. 
We claim 
\[
\lim_{\nu\to\infty}\varepsilon(\nu)=0.
\]
To see this, we note
\begin{align}
\label{R/a}
\lim_{k\to\infty} \frac{R(k)}{a_{+}(k)}=0,\ \lim_{k\to\infty} \frac{R(k)}{a_{-}(k)}=0.
\end{align}
In fact, for sufficiently large $k\in\mathbb{N}$,
the positive valued sequences 
$a_{-}=\{a_{-}(k)\}_{k\geq \nu}$ and $a_{+}=\{a_{+}(k)\}_{k\geq \nu}$ are monotone increasing by Assumption \ref{ass:assume1}.
Hence we get (\ref{R/a}) by Assumption \ref{ass:assume3}.
Therefore, $\lim_{\nu\to\infty}\varepsilon(\nu)=0$ follows 
from (\ref{R/a}) and Assumptions \ref{ass:assume2},\ref{ass:assume3}.

This subsection is devoted to proving the following proposition:
\begin{prop}
\label{prop:key-ineq}
Given a quadruple $(a_{-},a_{+},r,R)$ 
satisfying Assumptions \ref{ass:assume1}--\ref{ass:assume3}
and $\nu\in\mathbb{N}$,
let $j$, $\rho$, and $F_{\nu}^{\infty}$ be as in Notation \ref{notation-j-rho-F}.
Suppose that the inequalities \eqref{ineq:assume1} hold
for any integer $k\geq \nu-1$ and that $\varepsilon(\nu)<1$.
Let $\gamma\neq e$ be an arbitrary element of $F_{\nu}^{\infty}$
and $m:=\ell(\gamma)$ the word length of $\gamma$.
We write $\gamma=\gamma_{k_{1}}^{s_{1}}\ldots\gamma_{k_{m}}^{s_{m}}$ 
for the reduced expression where
$s_{1},\ldots,s_{m}\in\{1,-1\}$ and $k_{1},\ldots,k_{m}\geq \nu$.
Then the following inequalities hold:
\begin{enumerate}[label=(\arabic*)]
\item \label{prop:diffofmu}
$\displaystyle \left| \| j(\gamma) \| - \| \rho(\gamma) \| - 2\sum_{i=1}^{\ell(\gamma)} 
\log\frac{R(k_i)}{r(k_i)} \right| 
\leq \ell(\gamma)\varepsilon(\nu)$.
\item \label{prop:upper-bound-j}
$\displaystyle \|j(\gamma)\|\leq \left(\sum_{i=1}^{\ell(\gamma)}2\log\frac{a_{+}(k_{i})a_{-}(k_{i})}{r(k_{i})}\right) +
\ell(\gamma)\varepsilon(\nu)$.
\end{enumerate}
\end{prop}

\begin{rem}
Gu\'{e}ritaud-Kassel{ \cite[Sect.\ 10.1]{GuKa17}} gave 
an upper bound of 
\[
\left| \| j(\gamma) \| - \| \rho(\gamma) \| - 2\sum_{i=1}^{\ell(\gamma)} 
\log\frac{R(k_i)}{r(k_i)} \right|
\]
for the quadruple $(a_{-}(k),a_{+}(k),r(k),R(k)) = (k^2,k^2 + k,1,\log k)$,
see Table \ref{table-quadruple} (2).
However, since the explanation given there was not clear to the author,
we take an alternative approach to prove \ref{prop:diffofmu} in our general setting 
where a quadruple $(a_{-},a_{+},r,R)$ is arbitrary 
subject to Assumptions \ref{ass:assume1}--\ref{ass:assume3}.
\end{rem}

We prepare some results needed for the proof of Proposition \ref{prop:key-ineq}.
\begin{lem}
\label{varepsilon-property}
Let $(a_{-},a_{+},r,R)$ and $\nu\in\mathbb{N}$ be as in Proposition \ref{prop:key-ineq}
and $B_{k}^{\pm}$ the half-disks defined in (\ref{half-disks}).
Then, the following assertions hold:
\begin{enumerate}[label=$(\arabic*)$]
    \item \label{varepsilon-property-a-R}
    $a_{-}(k)>\max\{1,R(k)\}$ for any integer $k\geq \nu$.
    \item \label{varepsilon-property-i-B}
    $\displaystyle \sqrt{-1}\notin
   \bigcup_{k\geq\nu}(B_{k}^{-}\cup B_{k}^{+})$ in the upper-half plane $\mathbb{H}^{2}$.
    \item \label{varepsilon-property-lower-bound}
    The following inequalities hold:
    \begin{enumerate}[label=$(3\alph*)$]
    \item 
    $\displaystyle \varepsilon(\nu)\geq \max_{k\geq\nu}\left\{
    6\log\left(1+e^{-2d_{\mathbb{H}^{2}}(B_{k}^{+},\sqrt{-1})}\right),\ 6\log\left(1+e^{-2d_{\mathbb{H}^{2}}(B_{k}^{-},\sqrt{-1})}\right)\right\}$.
    \item $\displaystyle \varepsilon(\nu)\geq 
    \max_{k\geq\nu}\left\{12\left|\log\left(1+\frac{ R(k)}{a_{-}(k)}\right)\right|,\ 12\left|\log\left(1-\frac{ R(k)}{a_{-}(k)}\right)\right|\right\}$.
    \item 
    $\displaystyle \varepsilon(\nu)\geq
    \max_{k\geq\nu}\left\{6\log\left(1+\frac{R(k)^{2}+1}{(a_{-}(k)-R(k))^{2}}\right)\right\}$.
    \item 
    $\displaystyle \varepsilon(\nu)\geq 
    \max\left\{4|\log(1+\eta(\nu))|,\ 
    4|\log(1-\eta(\nu))|\right\}$.
    \end{enumerate}
\end{enumerate}
\end{lem}

\begin{proof}
Take any integer $\nu$ 
such that the inequalities \eqref{ineq:assume1} hold for any integer $k\geq \nu-1$ and that $\varepsilon(\nu)<1$. 
Then, by $\varepsilon(\nu)<1$, we have $a_{-}(k)>24R(k)>3R(k)$ and $a_{-}(k)-R(k)>\sqrt{6(R(k)^{2}+1)}>\sqrt{6}>2$. Thus we get
\begin{align}
a_{-}(k) &>3R(k) \label{ineq:a-R-1} \\
a_{-}(k)-R(k)&>2. \label{ineq:a-R-2}
\end{align}
Using the inequality \eqref{ineq:a-R-2}, 
we can immediately deduce the assertion \ref{varepsilon-property-a-R}, 
and \ref{varepsilon-property-i-B} follows obviously from \ref{varepsilon-property-a-R}.

To prove the assertion \ref{varepsilon-property-lower-bound}, we apply the inequalities $t\geq \log(1+t)$ and $2s \geq |\log (1-s)|$ for $t\geq 0$ and $0\leq s\leq\frac{1}{2}$. These inequalities imply $(3b)$--$(3d)$ since we have $\displaystyle \frac{R(k)}{a_{-}(k)}<\frac{1}{2}$ (from \eqref{ineq:a-R-1}) and $\displaystyle \eta(\nu)(<\frac{\varepsilon(\nu)}{8})<\frac{1}{2}$.

Let $\delta\in\{+,-\}$.
To prove the inequality $(3a)$, we claim 
\begin{align}
\label{varepsilon-property-step1}
e^{-2d_{\mathbb{H}^{2}}(B_{k}^{\delta},\sqrt{-1})}< \frac{R(k)}{a_{-}(k)}.
\end{align}
If we could show this claim, then the inequality $(3a)$ is proved 
again by the inequality $t\geq \log(1+t)$ for $t\geq 0$.

We now prove (\ref{varepsilon-property-step1}).
Again by the inequalities \eqref{ineq:a-R-1} and \eqref{ineq:a-R-2}, we have
\begin{align}
\label{varepsilon-property-step2}
(a_{-}(k)-R(k))^{2}>2(a_{-}(k)-R(k))>a_{-}(k)+R(k).
\end{align}
We write $x_{\delta}(k)+\sqrt{-1}y_{\delta}(k)\in\mathbb{H}^{2}$ for 
the closest point of $B_{k}^{\delta}$ to $\sqrt{-1}$ with respect to the hyperbolic distance.
Obviously, we have
\[x_{\delta}(k)\geq a_{\delta}(k)-R(k),\ 
y_{\delta}(k)\leq R(k).
\]
Thus, by Lemma \ref{hypdist}, 
\[
2\cosh d_{\mathbb{H}^{2}}(B_{k}^{\delta},\sqrt{-1})
>\frac{x_{\delta}(k)^{2}}{y_{\delta}(k)}
\geq\frac{(a_{\delta}(k)-R(k))^{2}}{R(k)}
\geq \frac{(a_{-}(k)-R(k))^{2}}{R(k)},
\]
where the last inequality follows from Assumption \ref{ass:assume1}.
Hence,
since $e^{-d_{\mathbb{H}^{2}}(B_{k}^{\delta},\sqrt{-1})}\leq 1$, we have
\[
e^{d_{\mathbb{H}^{2}}(B_{k}^{\delta},\sqrt{-1})}
>\frac{(a_{-}(k)-R(k))^{2}}{R(k)}-1
>\frac{a_{-}(k)}{R(k)}
\]
where the second inequality follows from (\ref{varepsilon-property-step2}).
Therefore,
\[
e^{-2d_{\mathbb{H}^{2}}(B_{k}^{\delta},\sqrt{-1})}
\leq e^{-d_{\mathbb{H}^{2}}(B_{k}^{\delta},\sqrt{-1})}
<\frac{R(k)}{a_{-}(k)}.
\]
This proves (\ref{varepsilon-property-step1}) and thus the lemma holds.
\end{proof}

\begin{lem}
\label{real-im}
In the setting of Proposition \ref{prop:key-ineq},
the following assertions hold for both $\varphi=j$ or $\varphi=\rho$:
\begin{enumerate}[label=(\arabic*)]
\item
\label{real-im-item:i-B}
$\varphi(\gamma)\sqrt{-1}\in B_{k_{1}}^{s_{1}}$.
\item
\label{real-im-item:phi-d}
$\|\varphi(\gamma)\|\geq d_{\mathbb{H}^{2}}(B_{k_{1}}^{s_{1}},\sqrt{-1})$.
\item\label{real-im-item:re-a}
$\mathrm{Re}(\varphi(\gamma)\sqrt{-1})\geq a_{-}(k_{1})-R(k_{1})$.
\item\label{real-im-item:im-R}
$\mathrm{Im}(\varphi(\gamma)\sqrt{-1})\leq R(k_{1})$.
\item 
\label{real-im-item:log-cosh}
$\displaystyle \left|\|\varphi(\gamma)\|-\log(2\cosh\|\varphi(\gamma)\|)\right|
\leq\frac{1}{6}\varepsilon(\nu)$. 

\item\label{real-im-item:re-j/rho}
$\displaystyle\left|2\log \frac{\mathrm{Re}(j(\gamma)\sqrt{-1})}{
\mathrm{Re}(\rho(\gamma)\sqrt{-1})}\right|\leq
\frac{1}{3}\varepsilon(\nu)$.
\item\label{real-im-item:im-j/rho}
$\displaystyle
\left|-\log \frac{\mathrm{Im}(j(\gamma)\sqrt{-1})}{\mathrm{Im}(\rho(\gamma)\sqrt{-1})}
-2\sum_{i=1}^{m}\log \frac{R(k_{i})}{r(k_{i})}\right|\leq (\ell(\gamma)-1)\varepsilon(\nu)$.
\end{enumerate}
Here $\ell(\gamma)$ is the word length of $\gamma$, and 
we have used the convention that $B_{k}^{\pm1}=B_{k}^{\pm}$ and $a_{\pm1}=a_{\pm}$.
\end{lem}

\begin{proof}
By Lemma \ref{varepsilon-property} \ref{varepsilon-property-a-R},
we note that $\sqrt{-1}\notin \bigcup_{k\geq\nu}(B_{k}^{-}\cup B_{k}^{+})$. 
Hence the assertion \ref{real-im-item:i-B} follows from Corollary \ref{cor:discrete-free}, and thus we get \ref{real-im-item:re-a} and \ref{real-im-item:im-R}.
By Lemma \ref{hypdist}, \ref{real-im-item:phi-d} follows from \ref{real-im-item:i-B}.

We have
\begin{align}
\left|\|\varphi(\gamma)\|-\log(2\cosh\|\varphi(\gamma)\|)\right|
&=\log\left(1+e^{-2\|\varphi(\gamma)\|}\right)\nonumber\\
&\leq\log\left(1+e^{-2d_{\mathbb{H}^{2}}(B_{k_{1}}^{s_{1}},\sqrt{-1})}\right)
\leq\frac{1}{6}\varepsilon(\nu),
\end{align}
where the second and third inequalities follow from
Lemmas \ref{real-im} (2) and \ref{varepsilon-property}, respectively.
This proves \ref{real-im-item:log-cosh}.

By \ref{real-im-item:i-B}, we get
$\left| \mathrm{Re}(\varphi(\gamma)\sqrt{-1})-a_{s_{1}}(k_{1})\right|\leq R(k_{1})$.
Thus noting $a_{s_{1}}(k_{1})>R(k_{1})$ by $\varepsilon(\nu)<1$,
we have
\begin{align}
\label{re-phi}
\left|\log \frac{\mathrm{Re}(\varphi(\gamma)\sqrt{-1})}{a_{s_{1}}(k_{1})}\right|
\leq \max_{\delta=\pm1}\left\{\left|\log\left(1+\frac{\delta R(k_{1})}{a_{s_{1}}(k_{1})}\right)\right|\ \right\} 
\leq \frac{1}{12}\varepsilon(\nu)
\end{align}
by Lemma \ref{varepsilon-property}. 
Hence \ref{real-im-item:re-j/rho} follows.

In the following, we set for $k\in\mathbb{N}$
\begin{align*}
r_{\varphi}(k):=
\begin{cases}
r(k) & \text{ for } \varphi=j, \\
R(k) & \text{ for } \varphi=\rho.
\end{cases}
\end{align*}
Then, by (\ref{tau-property}),
\begin{align}
\label{im}
\mathrm{Im}(\varphi(\gamma_{k}^{s})z)=\frac{r_{\varphi}(k)^{2}\mathrm{Im}z}{|z-a_{-s}(k)|^{2}}
\text{ for } z\in\mathbb{H}^{2},\ k\in\mathbb{N},\text{ and }s=\pm 1.
\end{align}

Let us prove \ref{real-im-item:im-j/rho}. Define 
$\sigma_{0},\ldots,\sigma_{m}\in F_{\nu}^{\infty}$ 
by $\sigma_{i}=\gamma_{k_{i+1}}^{s_{i+1}}\ldots\gamma_{k_{m}}^{s_{m}}$
for $0\leq i<m$
and $\sigma_{m}=1$.
We note  $\sigma_{0}=\gamma$.
For $0\leq i\leq m$, we set
\begin{align*}
Q(\sigma_{i}):=
\frac{\mathrm{Im}(j(\sigma_{i})\sqrt{-1})}{\mathrm{Im}(\rho(\sigma_{i})\sqrt{-1})},\ 
D(\varphi(\sigma_{i})):=\left|\varphi(\sigma_{i})\sqrt{-1}-a_{-s_{i}}(k_{i})\right|.
\end{align*}
We claim:
\[
\left|\log \frac{Q(\sigma_{i})}{Q(\sigma_{i+1})}+2\log\frac{R(k_{i+1})}{r(k_{i+1})}\right|
= \left|2\log \frac{D(j(\sigma_{i+1}))}{D(\rho(\sigma_{i+1}))}\right|
\text{ for } 0\leq i<m,
\tag{$\mathrm{Q}_{i}$}
\]
\[
\left|2\log \frac{D(j(\sigma_{i}))}{D(\rho(\sigma_{i}))}\right|
\leq 
\begin{cases}
\varepsilon(\nu) & \text{ if } 0< i<m, \\
0 & \text{ if } i=m.
\end{cases}
\tag{$\mathrm{D}_{i}$}
\]
This claims imply \ref{real-im-item:im-j/rho}.
Indeed, we get \ref{real-im-item:im-j/rho} by summing up $(\mathrm{Q}_{i})$ and $(\mathrm{D}_{i})$ for all $i$
because $\log Q(\sigma_{m})=0$.

It remains to verify $(\mathrm{Q}_{i})$ and $(\mathrm{D}_{i})$.
Because $\sigma_{i}=\gamma_{k_{i+1}}^{s_{i+1}}\sigma_{i+1}$, we have
\begin{align*}
\mathrm{Im}(\varphi(\sigma_{i})\sqrt{-1})
=\frac{r_{\varphi}(k_{i+1})^{2}\mathrm{Im}(\varphi(\sigma_{i+1})\sqrt{-1})}
{D(\varphi(\sigma_{i+1}))^{2}}
\tag{$\mathrm{I}_{i}$}
\end{align*}
for $0\leq i <m$ by (\ref{im}). Thus ($\mathrm{Q}_{i}$) follows from the formula
\begin{align*}
\frac{Q(\sigma_{i})}{Q(\sigma_{i+1})}
=\left(\frac{D(\rho(\sigma_{i+1}))}{D(j(\sigma_{i+1}))}\right)^{2}
\left(\frac{r(k_{i+1})}{R(k_{i+1})}\right)^{2}.
\end{align*}

We observe that 
$(\mathrm{D}_{i})$ for $i=m$ is obvious because 
$D(j(\sigma_{m}))=D(\rho(\sigma_{m}))=|\sqrt{-1}-a_{-s_{m}}(k_{m})|$.
For $0<i<m$, by the triangle inequality, we have
\begin{align*}
\left|D(\varphi(\sigma_{i}))-\left|a_{s_{i+1}}(k_{i+1})-a_{-s_{i}}(k_{i})\right|\right|\leq
\left|\varphi(\sigma_{i})\sqrt{-1}-a_{s_{i+1}}(k_{i+1})\right|
\leq R(k_{i+1}) 
\end{align*}
since $\varphi(\sigma_{i})\sqrt{-1}\in B_{k_{i+1}}^{s_{i+1}}$ by \ref{real-im-item:i-B}.
Hence noting 
\[
|a_{s_{i+1}}(k_{i+1})-a_{-s_{i}}(k_{i})|\geq\eta(\nu)^{-1}R(k_{i+1})>R(k_{i+1})
\]
by $\varepsilon(\nu)<1$,
we obtain
\begin{align*}
\left|\log \frac{D(\varphi(\sigma_{i}))}{
\left|a_{s_{i+1}}(k_{i+1})-a_{-s_{i}}(k_{i})\right|}\right| 
&\leq
\max_{\delta=\pm1}\left\{\left|
\log\left(1+\frac{\delta R(k_{i+1})}{\left|a_{s_{i+1}}(k_{i+1})-a_{-s_{i}}(k_{i})
\right|}\right)\right|\right\} \nonumber \\
&\leq\max_{\delta=\pm1}\left\{\left|\log(1+\delta
\eta(\nu))\right|\right\} \nonumber \leq\frac{1}{4}\varepsilon(\nu).
\tag{$\mathrm{D}'_{i}$}
\end{align*}
The second and third inequalities follow from 
the definition (\ref{def-eta}) of $\eta(\nu)$ and 
Lemma \ref{varepsilon-property}, respectively.
Hence $\left|\log \frac{D(j(\sigma_{i}))}{D(\rho(\sigma_{i}))}\right|
\leq\frac{\varepsilon(\nu)}{2}$.
Thus ($\mathrm{D}_{i}$) holds for all $i$ and the proof of \ref{real-im-item:im-j/rho} is completed.
\end{proof}
We are ready to prove Proposition \ref{prop:key-ineq}.
\begin{proof}[Proof of Proposition \ref{prop:key-ineq}]
\begin{enumerate}[label=$(\arabic*)$]
\item 
Let $\varphi\colon F^{\infty}_{\nu}\rightarrow \mathrm{SL}(2,\mathbb{R})$ be either $j$ or $\rho$.

By Lemma \ref{real-im} \ref{real-im-item:log-cosh}, 
we get
\begin{align}
\label{step1}
\left|\|j(\gamma)\|-\|\rho(\gamma)\|-\log\frac{\cosh\|j(\gamma)\|}
{\cosh\|\rho(\gamma)\|}\right|
\leq\frac{1}{3}\varepsilon(\nu).
\end{align}

By Lemma \ref{hypdist},
\[
2\cosh\|\varphi(\gamma)\|=\frac{\mathrm{Re}(\varphi(\gamma)\sqrt{-1})^{2}+\mathrm{Im}(\varphi(\gamma)\sqrt{-1})^{2}+1}
{\mathrm{Im}(\varphi(\gamma)\sqrt{-1})}.
\] 
Therefore, we have
\begin{align}
\label{phi-re-im}
\begin{split}
\left|\log(2\cosh\|\varphi(\gamma)\|)-
2\log \mathrm{Re}(\varphi(\gamma)\sqrt{-1})
+\log \mathrm{Im}(\varphi(\gamma)\sqrt{-1})\right|\\
=
\log\left(1+\frac{\mathrm{Im}(\varphi(\gamma)\sqrt{-1})^2+1}
{\mathrm{Re}(\varphi(\gamma)\sqrt{-1})^2}\right)
\leq
\log\left(1+\frac{R(k_{1})^{2}+1}{(a_{-}(k_{1})-R(k_{1}))^{2}}\right)
\leq\frac{1}{6}\varepsilon(\nu).
\end{split}
\end{align}
The second and third inequalities follow from Lemma \ref{real-im} \ref{real-im-item:re-a}, \ref{real-im-item:im-R}
and Lemma \ref{varepsilon-property}, respectively.
Hence  
\begin{align}
\label{step2}
\left|\log\frac{\cosh\|j(\gamma)\|}
{\cosh\|\rho(\gamma)\|}
-2\log \frac{\mathrm{Re}(j(\gamma)\sqrt{-1})}{
\mathrm{Re}(\rho(\gamma)\sqrt{-1})}
+\log \frac{\mathrm{Im}(j(\gamma)\sqrt{-1})}{
\mathrm{Im}(\rho(\gamma)\sqrt{-1})}\right|
\leq\frac{1}{3}\varepsilon(\nu).
\end{align}
Summing up Lemma \ref{real-im} \ref{real-im-item:re-j/rho}, \ref{real-im-item:im-j/rho},
(\ref{step1}), and (\ref{step2}), we obtain
\ref{prop:diffofmu}.

\item 
We have 
\begin{align}
\label{upper-bound-j-step1}
\|j(\gamma)\|\leq 2\log \mathrm{Re}(j(\gamma)\sqrt{-1})
-\log \mathrm{Im}(j(\gamma)\sqrt{-1}) +\frac{1}{3}\varepsilon(\nu)
\end{align}
by summing up \ref{real-im-item:log-cosh} of Lemma \ref{real-im} and (\ref{phi-re-im}) for $\varphi=j$.
By (\ref{re-phi}), we have
\begin{align}
\label{upper-bound-j-step2}
2\log \mathrm{Re}(j(\gamma)\sqrt{-1})\leq 2\log a_{s_{1}}(k_{1}) + \frac{1}{6}\varepsilon(\nu).
\end{align}

We claim:
\begin{align}
\label{upper-bound-j-step3}
-\log \mathrm{Im}(j(\gamma)\sqrt{-1})+2\log a_{s_{1}}(k_{1})
\leq \left(\sum_{i=1}^{m}2\log\frac{a_{+}(k_{i})a_{-}(k_{i})}{r(k_{i})}\right)
+ \frac{m}{2}\varepsilon(\nu).
\end{align}
If we could show this claim, then \ref{prop:upper-bound-j} is proved by summing up 
(\ref{upper-bound-j-step1}), (\ref{upper-bound-j-step2}), and (\ref{upper-bound-j-step3})
since $m=\ell(\gamma)\geq1$.

It remains to verify (\ref{upper-bound-j-step3}).
As in the proof of Lemma \ref{real-im},
$\gamma=\gamma_{k_{1}}^{s_{1}}\cdots\gamma_{k_{m}}^{s_{m}}\in F_{\nu}^{\infty}$
defines a sequence of elements $\sigma_{0},\ldots,\sigma_{m}\in F_{\nu}^{\infty}$ 
by $\sigma_{i}=\gamma_{k_{i+1}}^{s_{i+1}}\ldots\gamma_{k_{m}}^{s_{m}}$
for $0\leq i<m$ with $\sigma_{0}=\gamma$ and $\sigma_{m}=1$.
We set $D(j(\sigma_{i})):=\left|j(\sigma_{i})\sqrt{-1}-a_{-s_{i}}(k_{i})\right|$.
Multiplying all ($\mathrm{I}_{i}$) for $0\leq i< m$, with $\varphi=j$, we get 
$\mathrm{Im}(j(\gamma)\sqrt{-1})=\prod_{i=1}^{m}(r(k_{i})/D(j(\sigma_{i})))^{2}$. Hence
\begin{align}
\label{upper-bound-j-step3-step0}
-\log\mathrm{Im}(j(\gamma)\sqrt{-1})=
2\sum_{i=1}^{m}\left(\log D(j(\sigma_{i}))-\log r(k_{i})\right).
\end{align} 
Since $\sigma_{m}=1$, we have 
\begin{align}
2\left(\log D(j(\sigma_{m}))-\log a_{-s_{m}}(k_{m})\right)&
=\log\left(1+\frac{1}{a_{-s_{m}}(k_{m})^{2}}\right)
\leq \frac{1}{a_{-s_{m}}(k_{m})^{2}}  \nonumber\\
&\leq \frac{3(R(k_{m})^{2}+1)}{(a_{-}(k_{m})-R(k_{m}))^{2}}
\leq \frac{1}{2}\varepsilon(\nu), \label{upper-bound-j-step3-step1}
\end{align}
where the third and fourth inequalities follow 
from Assumption \ref{ass:assume1} and
the definition (\ref{varepsilon-def}) of $\varepsilon(\nu)$,
respectively. 

Note $|s-t|\leq st$ for $s,t>1$ and
$a_{+}(k)>a_{-}(k)>1$ for all $k\geq\nu$
by \eqref{ineq:assume1} and Lemma \ref{varepsilon-property} \ref{varepsilon-property-a-R}.
Thus, by ($\mathrm{D}'_{i}$) for $1\leq i<m$, we have
\begin{align}
2\log D(j(\sigma_{i})) &\leq 2\log|a_{s_{i+1}}(k_{i+1})-a_{-s_{i}}(k_{i})|+\frac{1}{2}\varepsilon(\nu)
\nonumber \\
&\leq 2(\log a_{s_{i+1}}(k_{i+1})+\log a_{-s_{i}}(k_{i}))+\frac{1}{2}\varepsilon(\nu).
\label{upper-bound-j-step3-step2}
\end{align}
Now the inequality (\ref{upper-bound-j-step3}) follows from 
(\ref{upper-bound-j-step3-step0}),
(\ref{upper-bound-j-step3-step1}), and (\ref{upper-bound-j-step3-step2}).
Thus the proof of \ref{prop:upper-bound-j} is completed. 
\end{enumerate}
\end{proof}

\subsection{Proper discontinuity of the action of 
\texorpdfstring{$\Gamma_{\nu}(a_{-},a_{+},r,R)$}{}}

\label{subsection:proper}
In this section we give a necessary and sufficient condition for $\Gamma_{\nu}(a_{-},a_{+},r,R)$
to act properly discontinuously on $\mathrm{AdS}^{3}$.

The action of the Schottky-like discrete group $\Gamma_{\nu}(a_{-},a_{+},r,R)$ on 
$\mathrm{AdS}^{3}$ is not always properly discontinuous. We
give a necessary and sufficient condition for
this action to be properly discontinuous:
\begin{prop}
\label{proper}
Let $(a_{-},a_{+},r,R)$ be a quadruple of sequences as in Proposition \ref{prop:key-ineq}.
The action of $\Gamma_{\nu}(a_{-},a_{+},r,R)$ on $\mathrm{AdS}^{3}$ 
is properly discontinuous for sufficiently large $\nu\in\mathbb{N}$ if and only if
\begin{align}
\label{R/r}
\lim_{k \to \infty}\frac{R(k)}{r(k)} = \infty.
\end{align}
In this case, take $\nu\in\mathbb{N}$ as in Proposition \ref{prop:key-ineq} and assume that $\log (R(k)/r(k))>1$ for any integer $k\geq \nu$.
Then the action of $\Gamma_{\nu}(a_{-},a_{+},r,R)$
is properly discontinuous.
\end{prop}

Postponing the proof of Proposition \ref{proper},
we state its immediate consequences in Proposition \ref{Gamdisc}
and Lemma \ref{lem-p(x)} as below. First, 
since the group $\Gamma_{\nu}(a_{-},a_{+},r,R)$ 
is torsion-free, any properly discontinuous action is free, hence
we obtain:
\begin{prop}
\label{Gamdisc}
Let a quadruple $(a_{-},a_{+},r,R)$ and $\nu\in\mathbb{N}$ be as in Proposition \ref{proper}.
Assume the condition (\ref{R/r}). Then $\Gamma_{\nu}(a_{-},a_{+},r,R)$ 
is a discontinuous group for $\mathrm{AdS}^{3}$.
\end{prop}

\begin{ex}
All $(a_{-},a_{+},r,R)$ in Table \ref{table-quadruple}
apply to Proposition \ref{Gamdisc}.
\end{ex}

Let us prove Proposition \ref{proper}.

\begin{proof}[Proof of Proposition \ref{proper}]
Recall $\Gamma_{\nu}(a_{-},a_{+},r,R) = 
\{(j(\gamma), \rho(\gamma)) \mid \gamma \in F^{\infty}_{\nu}\}$.
The Kobayashi-Benoist properness criterion (Fact \ref{KoBe}) tells us that
$\Gamma_{\nu}(a_{-},a_{+},r,R)$ acts properly discontinuously on $\mathrm{AdS}^{3}$
if and only if 
\[
\forall C>0,\ \#\{\gamma\in F^{\infty}_{\nu}\mid |\|j(\gamma)\|-\|\rho(\gamma)\||<C\}
<\infty.
\]

We suppose that the action of $\Gamma_{\nu}(a_{-},a_{+},r,R)$ on $\mathrm{AdS}^{3}$
is properly discontinuous for any sufficiently large $\nu\in\mathbb{N}$.
Then $\lim_{k\to\infty}\left|\|j(\gamma_{k})\|-\|\rho(\gamma_{k})\|\right|=\infty$
by Fact \ref{KoBe}.
By Proposition \ref{prop:key-ineq} \ref{prop:diffofmu},
\[
2\log\frac{R(k)}{r(k)}\geq \left|\|j(\gamma_{k})\|-\|\rho(\gamma_{k})\|\right|-\varepsilon(\nu)
\]
for any $k\geq\nu$ and thus $\lim_{k\to\infty}R(k)/r(k)=\infty$.

Conversely, we suppose $\lim_{k\to\infty}R(k)/r(k)=\infty$.
Take $\nu\in\mathbb{N}$ as in Proposition \ref{prop:key-ineq}
and assume that $\log (R(k)/r(k))>1$ for any integer $k\geq \nu$.
Since $\varepsilon(\nu)<1$,
we note $\|j(\gamma)\|-\|\rho(\gamma)\|\geq 0$ by Proposition \ref{prop:key-ineq} \ref{prop:diffofmu}
for any $\gamma\in F^{\infty}_{\nu}$.
Assume that $\gamma\in F^{\infty}_{\nu}\smallsetminus\{e\}$ satisfies 
$|\|j(\gamma)\|-\|\rho(\gamma)\||<C$.
Let $m:=\ell(\gamma)$ be the word length of $\gamma$ and
we write $\gamma=\gamma_{k_{1}}^{s_{1}}\cdots\gamma_{k_{m}}^{s_{m}}$
for the reduced expression where
$s_{1},\ldots,s_{m}\in\{1,-1\}$ and $k_{1},\ldots,k_{m}\geq \nu$.
By Proposition \ref{prop:key-ineq} \ref{prop:diffofmu},  
\[
\ell(\gamma)>\ell(\gamma)\varepsilon(\nu)
\geq 2\sum_{i=1}^{\ell(\gamma)} \log\cfrac{R(k_i)}{r(k_i)}- 
(\| j(\gamma) \| - \| \rho(\gamma) \|)
> 2\ell(\gamma) - C.
\]
Hence $\ell(\gamma)<C$. Again by Proposition \ref{prop:key-ineq} \ref{prop:diffofmu}, we get
\[
\sum_{i=1}^{\ell(\gamma)} \log\frac{R(k_i)}{r(k_i)}<\frac{1}{2}(
\left(\| j(\gamma) \| - \| \rho(\gamma) \|\right)+ \ell(\gamma)\varepsilon(\nu))
<C
\]
and there are only finitely many $\gamma$ satisfying this inequality.
By Fact \ref{KoBe},
the action of $\Gamma_{\nu}(a_{-},a_{+},r,R)$ on $\mathrm{AdS}^{3}$
is properly discontinuous.
\end{proof}

\subsection{Sharpness of the \texorpdfstring{$\Gamma_{\nu}(a_{-},a_{+},r,R)$}{}-action}

\label{sharpsec}
The notion of sharpness was introduced in Kassel-Kobayashi \cite{KaKob16}, 
although the idea was already implicit
in \cite{Kob98}.
It is defined for a general homogeneous space of reductive type.
However, in this section, 
we explain it only for $\mathrm{AdS}^{3}$.
Moreover, we find a necessary and sufficient condition that the discontinuous group 
$\Gamma_{\nu}(a_{-},a_{+},r,R)$
for $\mathrm{AdS}^{3}$ is sharp.

The Cartan projection $\mu$ for the direct product group 
$G=\mathrm{SL}(2,\mathbb{R})\times\mathrm{SL}(2,\mathbb{R})$
is given by $\mu(g)=(\|\alpha\|,\|\beta\|)$ for $g=(\alpha,\beta)\in G$,
where we recall that $\|\cdot\|$ is the pseudo-distance in $\mathrm{SL}(2,\mathbb{R})$.
By Fact~\ref{KoBe}, a discrete subgroup 
$\Gamma$ of $G$ acts properly discontinuously on $\mathrm{AdS}^{3}$
if and only if $\mu(\Gamma)$ ``goes away from the line $x = y$ at infinity''.
The condition of  sharpness is stronger than the condition of proper discontinuity as 
Definition~\ref{sharp} below, and $\Gamma$ is sharp for $\mathrm{AdS}^{3}$
if $\mu(\Gamma)$ ``goes away from the line $x = y$ at infinity'' 
with a speed that is at least linear (\cite[p.\ 152]{KaKob16}) 
as in Figure \ref{fig-proper-sharp}.
See also \cite[Ch.\ 4, Fig.\ 1]{KaKob16} for the illustration of sharp actions in the general setting.

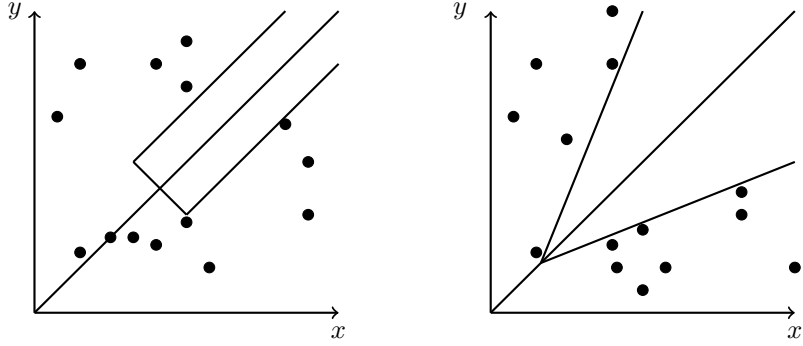
\begin{figure}
\begin{tikzpicture}[scale=1.0]

\draw [thick,->] (0.0,0.0) -- (4.0,0.0);
\node at (4.0,-0.26) {$x$};
\draw [thick,->] (0.0,0.0) -- (0.0,4.0);
\node at (-0.26,4.0) {$y$};

\draw [thick,->] (6.0,0.0) -- (10.0,0.0);
\node at (10.0,-0.26) {$x$};
\draw [thick,->] (6.0,0.0) -- (6.0,4.0);
\node at (5.6,4.0) {$y$};

\draw [thick,-] (0.0,0.0) -- (4.0,4.0);
\draw [thick,-] (1.3,2.0) -- (3.3,4.0);
\draw [thick,-] (2.0,1.3) -- (4.0,3.3);
\draw [thick,-] (2.0,1.3) -- (1.3,2.0);

\draw [thick,-] (6.0,0.0) -- (10.0,4.0);
\draw [thick,-] (6.66,0.66) -- (8.0,4.0);
\draw [thick,-] (6.66,0.66) -- (10.0,2.0);

\filldraw(1.0,1.0)circle (2pt);
\filldraw(1.3,1.0)circle (2pt);
\filldraw(2.0,3.0)circle (2pt);
\filldraw(3.6,1.3)circle (2pt);
\filldraw(3.6,2.0)circle (2pt);
\filldraw(3.3,2.5)circle (2pt);
\filldraw(2.3,0.6)circle (2pt);
\filldraw(1.6,3.3)circle (2pt);
\filldraw(2.0,3.6)circle (2pt);
\filldraw(1.6,0.9)circle (2pt);
\filldraw(0.6,0.8)circle (2pt);
\filldraw(2.0,1.2)circle (2pt);
\filldraw(0.6,3.3)circle (2pt);
\filldraw(0.3,2.6)circle (2pt);

\filldraw(7.0,2.3)circle (2pt);
\filldraw(7.66,0.6)circle (2pt);
\filldraw(8.0,0.3)circle (2pt);
\filldraw(9.3,1.3)circle (2pt);
\filldraw(9.3,1.6)circle (2pt);
\filldraw(10.0,0.6)circle (2pt);
\filldraw(8.3,0.6)circle (2pt);
\filldraw(7.6,3.3)circle (2pt);
\filldraw(7.6,4.0)circle (2pt);
\filldraw(7.6,0.9)circle (2pt);
\filldraw(6.6,0.8)circle (2pt);
\filldraw(8.0,1.1)circle (2pt);
\filldraw(6.6,3.3)circle (2pt);
\filldraw(6.3,2.6)circle (2pt);
\end{tikzpicture}
\caption{Properly discontinuous actions and sharp actions} \label{fig-proper-sharp}
\end{figure}

\begin{defi}[Kassel-Kobayashi {\cite[Def.\ 4.2]{KaKob16}}]
\label{sharp}
Let $c \in (0,1]$ and $C \geq 0$.
A discrete subgroup $\Gamma \subset G$ is called $(c,C)$-sharp for $\mathrm{AdS}^{3}$
if for any $\gamma=(\alpha,\beta) \in \Gamma$,
\begin{align}
\label{ineq:def-sharp}
|\|\alpha\| - \|\beta\|| \geq 
\sqrt{2}(c\sqrt{\|\alpha\|^{2}+\|\beta\|^{2}} - C).
\end{align}
If $\Gamma$ is $(c,C)$-sharp for some $c$ and $C$,
then $\Gamma$ is called sharp for $\mathrm{AdS}^{3}$.
\end{defi}

Kassel \cite[Thm.\ 0.2.13]{Ka09} and Gu\'{e}riataud-Kassel \cite[Thm.\ 1.8]{GuKa17} proved that any
finitely generated discontinuous group $\Gamma \subset G$
for $\mathrm{AdS}^{3}$ is sharp.
However, our discontinuous group
$\Gamma_{\nu}(a_{-},a_{+},r,R)$ is infinitely generated, and actually, 
$\Gamma_{\nu}(a_{-},a_{+},r,R)$ may not be sharp for $\mathrm{AdS}^{3}$.
The next proposition gives a necessary and sufficient condition for the sharpness of the action of 
$\Gamma_{\nu}(a_{-},a_{+},r,R)$ on $\mathrm{AdS}^{3}$ for $\nu\gg 0$:
\begin{prop}
\label{non-sharp}
Let $(a_{-},a_{+},r,R)$ be a quadruple of sequences 
satisfying Assumptions \ref{ass:assume1}--\ref{ass:assume3} and the condition (\ref{R/r}).
We set
\begin{align}
A\equiv A(a_{-},a_{+},r,R):=\liminf_{k \to \infty}\log\left(\frac{R(k)}{r(k)}\right)
\left({\log\frac{a_{-}(k)a_{+}(k)}{r(k)}}\right)^{-1}.
\end{align}
Then $0\leq A\leq 1$. Moreover, take $\nu\in\mathbb{N}$ as in Proposition \ref{proper}.
Then the following hold for the discontinuous group 
$\Gamma_{\nu}\equiv\Gamma_{\nu}(a_{-},a_{+},r,R)$ for $\mathrm{AdS}^{3}$:
\begin{enumerate}[label=$(\arabic*)$]
\item
if $c>A/\sqrt{2(1+(1-A)^{2})}$, then
$\Gamma_{\nu}$ is not $(c,C)$-sharp for any $C\geq 0$;
\item
Assume $0<c<A/\sqrt{2(1+(1-A)^{2})}$. Then
$\Gamma_{\nu}$ is $(c,0)$-sharp
if for any $k\geq \nu$, we have
\begin{align}
\label{ineq:c-sharp}
\frac{B+1}{A-B}<\log\frac{a_{-}(k)a_{+}(k)}{r(k)}
<\frac{2}{A+B}\log\frac{R(k)}{r(k)},
\end{align}
where $B\in(0,A)$ is defined by $c=B/\sqrt{2(1+(1-B)^{2})}$.
\end{enumerate}
In particular, $\Gamma_{\nu}$ is sharp for $\nu\gg 0$ if and only if $A\neq 0$.
\end{prop}

\begin{proof}
Take $\nu\in\mathbb{N}$ as in Proposition \ref{proper}.
Then $R(k)>r(k)$ and $a_{+}(k)>a_{-}(k)$ for any integer $k\geq\nu$
by the inequality \eqref{ineq:assume1}.
Moreover, $a_{-}(k)>\max\{R(k),1\}$ by Lemma \ref{varepsilon-property} \ref{varepsilon-property-a-R}.
Hence we have $a_{-}(k)a_{+}(k)>R(k)>r(k)$ 
and thus 
\begin{align}
\label{ineq:a+a-/r}
0<\log\left(\frac{R(k)}{r(k)}\right)
\left({\log\frac{a_{-}(k)a_{+}(k)}{r(k)}}\right)^{-1}<1.
\end{align}
In particular, we obtain $0\leq A\leq1$.
\begin{enumerate}[label=$(\arabic*)$]
\item
Recall that $\Gamma_{\nu}$ is generated by 
$\{(\alpha_{k},\beta_{k})\mid k=\nu,\nu+1,\ldots\}$.
We claim:
\begin{align}
\label{sharp-claim}
\liminf_{k\to\infty}
\frac{|\|\alpha_{k}\|-\|\beta_{k}\||}{\sqrt{\|\alpha_{k}\|^{2}+\|\beta_{k}\|^{2}}}
\leq 
\frac{A}{\sqrt{1+(1-A)^{2}}}.
\end{align}
If we could show this claim, then (1) is obvious.
Note $\lim_{k\to\infty}\sqrt{\|\alpha_{k}\|^{2}+\|\beta_{k}\|^{2}}=\infty$
since the map $\|\cdot\|\colon\mathrm{SL}(2,\mathbb{R})\rightarrow[0,+\infty]$ is proper
and since
$\{(\alpha_{k},\beta_{k})\mid k=\nu,\nu+1,\ldots\}$ is 
an infinite discrete subset of $G$.

We now prove (\ref{sharp-claim}).
By Proposition \ref{prop:key-ineq} \ref{prop:diffofmu}, for $k\geq \nu$,
\begin{align}
\label{non-sharp-step1}
|\|\alpha_{k}\|-\|\beta_{k}\||\leq 2\log\left(\frac{R(k)}{r(k)}\right)+\varepsilon(\nu).
\end{align}
Since $\alpha_k = r(k)^{-1}\begin{pmatrix}
a_{+}(k) & -(a_{-}(k)a_{+}(k) + r(k)^2) \\
1 & -a_{-}(k)
\end{pmatrix}$ by the definition (\ref{alpha-beta-def}), we have 
\begin{align}
\label{non-sharp-step2}
\|\alpha_{k}\|&\geq \log (2\cosh\|\alpha_{k}\|) -\frac{1}{6}\varepsilon(\nu) \nonumber \\
&\geq 2\log\left(\frac{a_{-}(k)a_{+}(k)}{r(k)}\right) - \frac{1}{6}\varepsilon(\nu) \nonumber \\
&\geq 2\log\left(\frac{a_{-}(k)a_{+}(k)}{r(k)}\right) - \varepsilon(\nu),
\end{align}
where the first and second inequalities follow from 
Lemma \ref{real-im} \ref{real-im-item:log-cosh} and the definition (\ref{def-pdist}), respectively. 
Similarly, we have
\begin{align}
\|\beta_{k}\| &\geq 2\log\left(\frac{a_{-}(k)a_{+}(k)}{R(k)}\right) - \varepsilon(\nu)
\nonumber \\
\label{non-sharp-step3}
&=2\left(\log\left(\frac{a_{-}(k)a_{+}(k)}{r(k)}\right)-\log\left(\frac{R(k)}{r(k)}\right)\right)
- \varepsilon(\nu).
\end{align}

Here we put 
\begin{align*}
a_{k} &:= \log\left(\frac{R(k)}{r(k)}\right)
\left({\log\frac{a_{-}(k)a_{+}(k)}{r(k)}}\right)^{-1}, \\ 
b_{k}&:= \frac{\varepsilon(\nu)}{2}\left({\log\frac{a_{-}(k)a_{+}(k)}{r(k)}}\right)^{-1}.
\end{align*}
Then, by (\ref{non-sharp-step1})--(\ref{non-sharp-step3}),
we have
\[
\liminf_{k\to\infty}
\frac{|\|\alpha_{k}\|-\|\beta_{k}\||}{\sqrt{\|\alpha_{k}\|^{2}+\|\beta_{k}\|^{2}}}
\leq 
\liminf_{k\to\infty}
\frac{a_{k} + b_{k}}
{\sqrt{(1-b_{k})^{2} + (1-a_{k}-b_{k})^{2}}}.
\]
Here $0<a_{k}<1$ by \eqref{ineq:a+a-/r} and 
$\lim_{k\to\infty}b_{k} = 0$ by the condition \eqref{R/r} since $0<b_{k} \leq \varepsilon(\nu)(2\log(R(k)/r(k)))^{-1}$.
Hence, by applying Lemma \ref{lemma:elementary-ineq} below
to $a_{k}$ and $b_{k}$, we obtain
the inequality (\ref{sharp-claim}).
Thus (1) is proved.

\item
Suppose $0<c<A/\sqrt{2(1+(1-A)^{2})}$. 
Take $\nu\in\mathbb{N}$ as in Proposition \ref{proper} 
and assume that the inequalities \eqref{ineq:c-sharp} hold for any integer $k\geq \nu$. 
Then we note $\varepsilon(\nu)<1$. Setting $2\xi:=A-B(>0)$, for any integer $k\geq \nu$, we have
\begin{align}
(B+\xi) \log\frac{a_{-}(k)a_{+}(k)}{r(k)} < \log \frac{R(k)}{r(k)}, \label{non-sharp-step4} \\
\ (B+1)\varepsilon(\nu)<2\xi\log\frac{a_{-}(k)a_{+}(k)}{r(k)} \label{non-sharp-step5},
\end{align}
by \eqref{ineq:c-sharp}. Let $\gamma\neq e$ be an arbitrary element of $F_{\nu}^{\infty}$,
and $m:=\ell(\gamma)$ the word length of $\gamma$. We write
$\gamma=\gamma_{k_{1}}^{s_{1}}\cdots\gamma_{k_{m}}^{s_{m}}$ for the reduced expression
where $s_{1},\ldots,s_{m}\in\{1,-1\}$ and $k_{1},\ldots,k_{m}\geq\nu$.
By Proposition \ref{prop:key-ineq} \ref{prop:upper-bound-j}, we have
\begin{align*}
B\|j(\gamma)\|&\leq 
\sum_{i=1}^{m} \left(2B\log\left(\frac{a_{-}(k_{i})a_{+}(k_{i})}{r(k_{i})}\right)+
B\varepsilon(\nu)\right) \\
&< \sum_{i=1}^{m}\left(2\log\left(\frac{R(k_{i})}{r(k_{i})}\right)
-\left(2\xi\log\left(\frac{a_{-}(k)a_{+}(k)}{r(k)}\right) - B\varepsilon(\nu)\right)\right)\\
&< \sum_{i=1}^{m} \left(2\log\left(\frac{R(k_{i})}{r(k_{i})}\right) - \varepsilon(\nu)\right)
\leq \|j(\gamma)\|-\|\rho(\gamma)\|.
\end{align*}
Here the second, third, and fourth inequalities follow 
from (\ref{non-sharp-step4}), (\ref{non-sharp-step5}),
and Proposition \ref{prop:key-ineq} \ref{prop:diffofmu}, respectively. 
Then $\|\rho(\gamma)\|< (1-B)\|j(\gamma)\|$ and thus
\begin{align*}
\frac{\|j(\gamma)\|-\|\rho(\gamma)\|}{\sqrt{2(\|j(\gamma)\|^{2}+\|\rho(\gamma)\|^{2})}} 
> \frac{B}{\sqrt{2(1+(1-B)^{2})}}
=c. 
\end{align*}
Hence $\Gamma_{\nu}$ is $(c,0)$-sharp, which proves (2).
\end{enumerate}
\end{proof}

In the above proof, we have used the following elementary lemma:
\begin{lem}
\label{lemma:elementary-ineq}
Let $a=\{a_{k}\}_{k\in\mathbb{N}}\in [0,1]^{\mathbb{N}}$ and $b=\{b_{k}\}_{k\in\mathbb{N}}\in [0,\infty]^{\mathbb{N}}$. 
If $\displaystyle \lim_{k\to\infty} b_{k} =0$,
then we have 
\[
\liminf_{k\to\infty} \frac{a_{k} + b_{k}}
{\sqrt{(1-b_{k})^{2} + (1-a_{k}-b_{k})^{2}}}
=\frac{\liminf_{k\to\infty}a_{k}}{\sqrt{1+(1-\liminf_{k\to\infty}a_{k})^{2}}}.
\]
\end{lem}

\begin{proof}
We note that 
the continuous function $f(x):= x(1+(1-x)^{2})^{-\frac{1}{2}}$ is monotone increasing on the interval $[0,2]$. 
In particular, we have
$\displaystyle \liminf_{k\to\infty} f(a_{k})
= f(\liminf_{k\to\infty} a_{k})$.
Hence, to prove our claim, it suffices to show
\begin{align}
\label{eq:liminf1}
\liminf_{k\to\infty} \frac{a_{k} + b_{k}}
{\sqrt{(1-b_{k})^{2} + (1-a_{k}-b_{k})^{2}}}
&= \liminf_{k\to\infty} f(a_{k}).
\end{align}

For this purpose, take an arbitrary $0<\epsilon<1$.
By $\lim_{k\to\infty} b_{k} =0$, we have 
$0\leq b_{k} \leq \epsilon(1+\epsilon)^{-1}(<1)$
for any sufficiently large integer $k$.
Then the inequalities
\begin{align}
\label{ineq:liminf-1}
f(a_{k}+b_{k}) \leq 
\frac{a_{k} + b_{k}}
{\sqrt{(1-b_{k})^{2} + (1-a_{k}-b_{k})^{2}}}
\leq (1+\epsilon)f(a_{k}+b_{k}).
\end{align}
holds.
Indeed, since $(1+\epsilon)^{-1}\leq 1-b_{k}\leq 1$, we have 
\begin{align*}
\frac{a_{k} + b_{k}}
{\sqrt{(1-b_{k})^{2} + (1-a_{k}-b_{k})^{2}}}
\geq \frac{a_{k} + b_{k}}
{\sqrt{1 + (1-a_{k}-b_{k})^{2}}} 
= f(a_{k}+b_{k}),
\end{align*}
and 
\begin{align*}
\frac{a_{k} + b_{k}}
{\sqrt{(1-b_{k})^{2} + (1-a_{k}-b_{k})^{2}}}
&\leq \frac{a_{k} + b_{k}}
{\sqrt{(1+\epsilon)^{-2} + (1+\epsilon)^{-2}(1-a_{k}-b_{k})^{2}}} \\
&= (1+\epsilon)f(a_{k}+b_{k}).
\end{align*}
Thus the inequality \eqref{ineq:liminf-1} holds.

Here we take $M>0$ such that $f(y)\leq f(x) + M(y-x)$ for any $0\leq x\leq y\leq 2$.
Then we have $f(a_{k}) \leq f(a_{k}+b_{k})\leq f(a_{k})+Mb_{k}$ since $0\leq a_{k}\leq a_{k}+b_{k}\leq 2$.
By combining this with the inequalities \eqref{ineq:liminf-1}, we have
\[
f(a_{k}) \leq 
\frac{a_{k} + b_{k}}
{\sqrt{(1-b_{k})^{2} + (1-a_{k}-b_{k})^{2}}}
\leq (1+\epsilon)(f(a_{k}) + Mb_{k}).
\]
Taking the limit as $k\to\infty$, we obtain
\[
\liminf_{k\to\infty} f(a_{k}) \leq 
\liminf_{k\to\infty} \frac{a_{k} + b_{k}}
{\sqrt{(1-b_{k})^{2} + (1-a_{k}-b_{k})^{2}}}
\leq (1+\epsilon)\liminf_{k\to\infty}f(a_{k}).
\]
Since $0<\epsilon<1$ is arbitrary,  
the equality \eqref{eq:liminf1} holds.
This proves our claim.
\end{proof}

\begin{rem}
\label{critical-value}
In Proposition \ref{non-sharp}, we did not treat the case where 
$c$ takes the critical value
$c(a_{-},a_{+},r,R):=A(2(1+(1-A)^{2}))^{-\frac{1}{2}}$.
At this critical value, the discontinuous group
$\Gamma_{\nu}(a_{-},a_{+},r,R)$ may be $(c,C)$-sharp for all $C\geq 0$,
and may not be $(c,C)$-sharp for all $C\geq 0$.
We give such examples below.

Let $\delta=\pm 1$ and $b\geq 0$.
We define the following quadruple $(a_{-},a_{+},r,R_{\delta})$:
\[
a_{-}(k):=e^{\frac{b}{2}(k-\frac{1}{4})+1},\ a_{+}(k):=e^{\frac{b}{2}(k+\frac{1}{4})+1},\  
r(k):=e^{-k},\ R_{\delta}(k):=(\log k)^{\delta}
\] 
for $k\geq 2$. This quadruple satisfies Assumptions \ref{ass:assume1}--\ref{ass:assume3} and the condition (\ref{R/r}).
Then the critical value $c\equiv c(a_{-},a_{+},r,R_{\delta})$
is $(2(b^{2}+(b+1)^{2}))^{-\frac{1}{2}}$ because 
\[
\log\frac{R_{\delta}(k)}{r(k)}=k+\delta\log\log k,\ 
\log \frac{a_{-}(k)a_{+}(k)}{r(k)}=(b+1)k+2.
\]
Take $\nu\in\mathbb{N}$ as in Proposition \ref{proper}. 
Let $\Gamma_{\nu,\delta}$ be 
the discontinuous group 
$\Gamma_{\nu}(a_{-},a_{+},r,R_{\delta})$ for $\mathrm{AdS}^{3}$.
Then, the following hold:
\begin{enumerate}[label=$(\arabic*)$]
\item
if $\delta=1$, then $\Gamma_{\nu,\delta}$ is $(c,0)$-sharp for $\nu\gg0$, 
and thus $(c,C)$-sharp for all $C\geq 0$;
\item
if $\delta=-1$, then $\Gamma_{\nu,\delta}$ is not $(c,C)$-sharp for all $C\geq 0$.
\end{enumerate}
Indeed, suppose $\delta=1$. Then
we may and do take $\nu\gg 0$ such that for all $k\geq\nu$,
\[
\frac{1}{b+1}\left(2\log\left(\frac{a_{-}(k)a_{+}(k)}{r(k)}\right) +\varepsilon(\nu)\right)
\leq 2\log\left(\frac{R_{\delta}(k)}{r(k)}\right)-\varepsilon(\nu).
\]
Then the inequality $(b+1)^{-1}\|j(\gamma)\|\leq \|j(\gamma)\|-\|\rho(\gamma)\|$
for any $\gamma\in F_{\nu}^{\infty}$ follows
from \ref{prop:diffofmu} and \ref{prop:upper-bound-j} of Proposition \ref{prop:key-ineq} 
by an argument similar to the proof of Proposition \ref{non-sharp} (2).
Hence $\|\rho(\gamma)\|\leq b(1+b)^{-1}\|j(\gamma)\|$
and thus (1) holds since 
\[
\frac{\|j(\gamma)\|-\|\rho(\gamma)\|}{\sqrt{2(\|j(\gamma)\|^{2}+\|\rho(\gamma)\|^{2})}} 
\geq \frac{1}{\sqrt{2(b^{2}+(1+b)^{2})}}=c.
\]

On the other hand, suppose $\delta=-1$.
Since $\varepsilon(\nu)<1$,
we have
$|\|\alpha_{k}\|-\|\beta_{k}\||\leq 2(k-\log\log k)+1$,
$\|\alpha_{k}\|\geq 2(b+1)k$, and $\|\beta_{k}\|\geq 2bk$
for any integer $k\geq\nu$ 
by (\ref{non-sharp-step1}), (\ref{non-sharp-step2}), and (\ref{non-sharp-step3}),
respectively.
Thus (2) follows readily from 
\[
\liminf_{k\to\infty}\left(\sqrt{2}c\sqrt{\|\alpha_{k}\|^{2}+
\|\beta_{k}\|^{2}}-|\|\alpha_{k}\|-\|\beta_{k}\||\right)
\geq \liminf_{k\to\infty}\left(2\log\log k-1\right)=\infty.
\]
\end{rem}

\begin{ex}
\label{ex-gamma-non-sharp}
The discontinuous groups $\Gamma_{\nu}(a_{-},a_{+},r,R)$ associated to the quadruples
$(a_{-},a_{+},r,R)$ in Table \ref{table-quadruple} are all non-sharp 
by Proposition \ref{non-sharp}.
\end{ex}

\section{Construction of \texorpdfstring{$\Gamma$}{} with large counting}
\label{countingbelowsec}
Let $G=\mathrm{SL}(2,\mathbb{R})\times \mathrm{SL}(2,\mathbb{R})$.
In this section, we explain the construction of a discontinuous group $\Gamma \subset G$ for which the asymptotic growth of the counting $N_{\Gamma}(x,R)$ is
as rapid as we wish, and thus complete the proof of Theorem \ref{countingbelow}.

\subsection{Construction of $\Gamma$ and a lower bound of 
$N_{\Gamma}(o,R)$}

To construct a discontinuous group satisfying the desired property,
we use the following lemma which is
a special case of Proposition \ref{Gamdisc}:

\begin{lem}
\label{lem-p(x)}
Let $p(t)$ be a $C^{2}$-function 
defined for sufficiently large $t$, say $t\geq \nu_{0}$,
such that $p(t)>0$, $p'(t)>0$, $p''(t)<0$ and $\lim_{t\to\infty}p(t)=\infty$,
and $q(t)$ a continuous function satisfying $\lim_{t\to\infty}q(t)=\infty$.
Let $\Gamma_{\nu}(p,q)$ denote the group $\Gamma_{\nu}(a_{-},a_{+}.r,R)$
associated to the sequences
\begin{align}
\label{quadruple-p(x)-q(x)}
a_{-}(k):=p(k),\ a_{+}(k):=p(k+\frac{1}{2}),\ r(k):=\frac{p'(k+1)}{p(k)q(k)},\ 
R(k):=\frac{p'(k+1)}{q(k)}
\end{align}
for $k\geq\nu_{0}$.
Then $\Gamma_{\nu}(p,q)$ is a discontinuous group for $\mathrm{AdS}^{3}$
if $\nu\gg 0$.
\end{lem}

\begin{proof}[Proof of Lemma \ref{lem-p(x)}]
Let us check that our $(a_{-},a_{+},r,R)$ satisfies Assumptions \ref{ass:assume1}--\ref{ass:assume3} and the condition \eqref{R/r}.
If we could check this, it follows from Proposition \ref{Gamdisc} that
$\Gamma_{\nu}(p,q)$ is a discontinuous group for $\mathrm{AdS}^{3}$
for sufficiently large $\nu\in\mathbb{N}$,
which proves our claim.

Since $\lim_{t\to\infty}q(t)=\infty$, 
we may assume $q(t)>4$ for any $t\geq \nu_{0}$.
We note that  
the derivative $p'(x)$ is monotone decreasing because 
$p''(x)<0$.
By the mean value theorem,
for any integer $k\geq\nu_{0}$, we have
\[
a_{+}(k)-a_{-}(k)=p(k+\frac{1}{2})-p(k)\geq \frac{p'(k+1)}{2}>2R(k),
\]
\[
a_{-}(k+1)-a_{+}(k)=p(k+1)-p(k+\frac{1}{2})
\geq \frac{p'(k+1)+p'(k+2)}{4}> R(k)+R(k+1).
\]
Thus Assumption \ref{ass:assume1} is verified.
Assumption \ref{ass:assume2} is obviously satisfied.

Take any $k,\ell\geq \nu_{0}$ and $\delta,\epsilon\in\{+,-\}$
such that $(k,\delta)\neq(\ell,\epsilon)$.
Since the function $p(x)$ is monotone increasing,
\begin{align*}
|a_{\delta}(k)-a_{\epsilon}(\ell)| 
&\geq \min\left\{p(k)-p(k-\frac{1}{2}),p(k+\frac{1}{2})-p(k),p(k+1)-p(k+\frac{1}{2})\right\} \\
&\geq\frac{p'(k+1)}{2},
\end{align*}
where the second inequality follows from the mean value theorem. 
Hence we get 
\[
\frac{R(k)}{|a_{\delta}(k)-a_{\epsilon}(\ell)|}
\leq \frac{2}{q(k)}
\]
and thus
Assumption \ref{ass:assume3} is verified since $\lim_{k\to\infty}q(k)=\infty$. 
The condition \eqref{R/r} is also satisfied since 
$\lim_{k\to\infty} R(k)/r(k) = \lim_{k\to\infty} p(k) = \infty$.
This completes the proof.
\end{proof}

The next proposition gives a lower bound of the counting $N_{\Gamma}(o,R)$ of the orbit through the base point $o$
of the discontinuous group $\Gamma=\Gamma_{\nu}(p,q)$.

\begin{prop}
\label{countingbelow-p(x)-q(x)}
Let $\Gamma_{\nu}\equiv \Gamma_{\nu}(p,q)$ be the discontinuous group
for $\mathrm{AdS}^{3}$ associated to a pair of functions $(p(t),q(t))$ 
as in Lemma \ref{lem-p(x)}.
Here we have taken $\nu\in\mathbb{N}$ as in Proposition \ref{proper}. 
Then   
\[
N_{\Gamma_{\nu}}(o,4\log p(t))\geq t-\nu \text{ for all } t\geq \nu.
\]
\end{prop}
\begin{proof}
Since we have taken $\nu\in\mathbb{N}$ as in Proposition \ref{proper},
we get $\log p(k)=\log (R(k)/r(k))>1$ for any $k\geq\nu$
and thus $p(k)\geq 2$.
We also recall from Definition \ref{def-gam} that the group 
$\Gamma_{\nu}$ is generated by 
$\{(\alpha_{k},\beta_{k})\mid k=\nu,\nu+1,\ldots\}$.

By the definitions (\ref{alpha-beta-def}) and (\ref{quadruple-p(x)-q(x)}), 
for $k\geq\nu$, we have
\begin{align*}
\alpha_{k}^{-1}\beta_{k}=
\begin{pmatrix}
p(k)^{-1} & p(k)^{2}-1 \\
0 & p(k)
\end{pmatrix},
\end{align*}
whence the pseudo-distance of 
$\alpha_{k}^{-1}\beta_{k}\in\mathrm{SL}(2,\mathbb{R})$
is computed by (\ref{def-pdist}):
\[
2\cosh \|\alpha_{k}^{-1}\beta_{k}\|
=p(k)^{4}-p(k)^{2}+1+p(k)^{-2} < p(k)^{4}.
\] 
We then observe for any $R>0$:
\[
\Gamma_{\nu}o \cap B(R)
\supset
\{\alpha_{k}^{-1}\beta_{k}
\mid \|\alpha_{k}^{-1}\beta_{k}\| \leq R\}
\supset
\{\alpha_{k}^{-1}\beta_{k}\mid p(k)^{4}\leq e^{R}\}
\]
since $2\cosh R > e^{R}$. Since $\Gamma_{\nu}$ acts freely on $\mathrm{AdS}^{3}$,
we deduce
\[
N_{\Gamma_{\nu}}(o,R)\geq 
\#\{k\in\mathbb{N}\mid \nu\leq k \text{ and } 4\log p(k) \leq R\}.
\]
Recall that $p(t)$ is monotone increasing. Hence we conclude 
$N_{\Gamma_{\nu}}(o,4\log p(t))\geq \#\{k\in\mathbb{N} \mid \nu\leq k\leq t\}>t-\nu$
for all $t\geq\nu$.
\end{proof}

\begin{ex}
\label{expexp}
Let $\Gamma_{\nu}\equiv\Gamma_{\nu}(a_{-},a_{+},r,R)$ be 
the discontinuous group for $\mathrm{AdS}^{3}$
associated to the quadruple $(a_{-},a_{+},r,R)$ in Table \ref{table-quadruple} (3).
Here we have taken $\nu\in \mathbb{N}$ as in Proposition \ref{proper}.
By an argument similar to the proof of Proposition \ref{countingbelow-p(x)-q(x)}, 
we have
\[N_{\Gamma_{\nu}}(o,R) \geq \exp(e^{\frac{R}{4}})- \nu
\text{ for all } R\geq 4\log\log\nu.\]
\end{ex}

\subsection{Proof of Theorem \ref{countingbelow}}
Let us prove Theorem \ref{countingbelow}.
We begin with a lemma which reduces the estimate of $N_{\Gamma}(x,R)$
to the case $x=o$.
\begin{lem}
\label{triangle}
Let $\Gamma$ be a discontinuous group for $\mathrm{AdS}^{3}$.
For $g=(\alpha,\beta)\in G$,
$x\in\mathrm{AdS}^{3}$, and $R>0$, we have
\[
N_{\Gamma}(gx,R-(\|\alpha\|+\|\beta\|))
\leq N_{g^{-1}\Gamma g}(x,R)\leq N_{\Gamma}(gx,R+(\|\alpha\|+\|\beta\|)).
\]
\end{lem}
\begin{proof}
By Lemma \ref{prop-pdist}, we have
$\|gx\|=\|\alpha x\beta^{-1}\|\leq \|x\| + (\|\alpha\|+\|\beta\|)$
and thus $gB(R) \subset B(R + (\|\alpha\|+\|\beta\|))$.
Similarly, we have $g^{-1}B(R-(\|\alpha\|+\|\beta\|))\subset B(R)$. Thus we obtain
\[
B(R-(\|\alpha\|+\|\beta\|))\subset gB(R) \subset B(R + (\|\alpha\|+\|\beta\|)).
\]
The assertion follows immediately from these inclusion relations.
\end{proof}

We also use the following elementary lemma.
\begin{lem}
\label{atarimae}
Given an increasing function $f \colon \mathbb{R} \to \mathbb{R}_{>0}$,
there exists a strictly increasing, convex, and $C^2$ function 
$F\colon \mathbb{R}_{\geq 0} \rightarrow \mathbb{R}_{>0}$ 
satisfying $F > f$ on $\mathbb{R}_{\geq 0}$ and $\lim_{x \to \infty}F(x) = \infty$.
\end{lem}
\begin{proof}[Proof of Lemma \ref{atarimae}]
We can construct such a function $F$ on $[0,n]$ by induction on $n\in\mathbb{N}$,
applying the following obvious claim to 
\[
(s,t,a,b,c) = 
\begin{cases}
(0,1,f(1),1,f(2)), & \\
(n,n+1,F(n),\lim_{x\to n-0}F'(x),f(n+2) + n) & (n\geq 1).
\end{cases}
\]

\textit{Claim}.
Given an interval $[s,t]$ and constants $a,b,c > 0$,
there exists a convex, strictly increasing, and $C^2$ function 
$h\colon[s,t]\rightarrow\mathbb{R}_{>0}$ satisfying
\[
h(s)=a,\ \lim_{x \to s+0}h'(x)=b,\ h(t)\geq c,\ \lim_{x\to t-0}h'(x) < \infty.
\]
\end{proof}

We are ready to prove Theorem \ref{countingbelow}.
\begin{proof}[Proof of Theorem \ref{countingbelow}]
Since the $G$-action on $\mathrm{AdS}^{3}$ is transitive,
we may and do assume $x=o$ by Lemma \ref{triangle}.
Moreover, by Lemma \ref{atarimae}, it suffices to consider the case where $f(t)$ is
a $C^{2}$-function such that $f(t)>0$, $f'(t)>0$, 
$f''(t)>0$ for any $t>0$ and
$\lim_{t\to\infty}f(t)=\infty$.

Let $p(t)$ be the inverse function of $sf(s)$
defined for $t\gg 0$.
Then $p(t)$ is a $C^{2}$-function 
such that $p(t)>0$, $p'(t)>0$, $p''(t)<0$ for $t\gg0$, 
and $\lim_{t\to\infty}p(t)=\infty$.
Take an arbitrary continuous function $q(t)$ satisfying $\lim_{t\to\infty}q(t)=\infty$.
Then
$\Gamma_{\nu}(p,q)$ is a discontinuous group 
for $\mathrm{AdS}^{3}$ if $\nu\gg 0$ by Lemma \ref{lem-p(x)}.
Take $\nu\in\mathbb{N}$ as in Proposition \ref{proper} and 
set $\Gamma\equiv\Gamma_{f,o}:=\Gamma_{\nu}(p,q)$.
By Proposition \ref{countingbelow-p(x)-q(x)}, 
$N_{\Gamma}(o,R)\geq e^{R/4}f(e^{R/4})-\nu$ for all $R\geq 4\log p(\nu)$. 
Thus $\lim_{R \to \infty}N_{\Gamma}(o,R)/f(R) = \infty$
and the proof of Theorem \ref{countingbelow} is completed.
\end{proof}


\section{Application to the spectral analysis}
\label{dspec}

In this section, we complete the proofs of 
Theorems \ref{countingabove-intro} and \ref{non-sharp-spectrum-intro}.

Associated to a quadruple $(a_{-},a_{+},r,R)$ of sequences,
we have defined in Section \ref{Gam} the subgroup 
$\Gamma_{\nu}\equiv\Gamma_{\nu}(a_{-},a_{+},r,R)$ 
of $G=\mathrm{SL}(2,\mathbb{R})\times \mathrm{SL}(2,\mathbb{R})$,
and proved that $\Gamma_{\nu}$ is a discontinuous group
for $X=\mathrm{AdS}^{3}$ when $\nu\in\mathbb{N}$ is sufficiently large
if Assumptions \ref{ass:assume1}--\ref{ass:assume3} and the condition (\ref{R/r}) are satisfied.
Then the quotient space
$X_{\Gamma_{\nu}}:=\Gamma_{\nu}\backslash X$
admits an anti-de Sitter structure via the covering $X\rightarrow X_{\Gamma_{\nu}}$.
Let $\square_{X_{\Gamma_{\nu}}}$ be the Laplacian on $X_{\Gamma_{\nu}}$,
which is not an elliptic operator, but a hyperbolic operator 
because $X_{\Gamma_{\nu}}$ is a Lorentzian manifold.
Our interest here is the discrete spectrum 
$\mathrm{Spec}_{d}(\square_{X_{\Gamma_{\nu}}})$ of 
the Laplacian $\square_{X_{\Gamma_{\nu}}}$.

In this section, we give a sufficient condition on 
$(a_{-},a_{+},r,R)$ for the exponential growth condition (\ref{count})
of the counting $N_{\Gamma_{\nu}}(x,R)$.
By the criterion of sharpness in Section \ref{sharpsec}, we give non-sharp 
$\Gamma_{\nu}$
such that the counting $N_{\Gamma_{\nu}}(x,R)$ has at most an exponential growth
uniformly on $x\in X$.
Moreover, we use this counting result to 
construct an infinite subset of  
$\mathrm{Spec}_{d}(\square_{X_{\Gamma_{\nu}}})$ for such $\Gamma_{\nu}$.
We show:
\begin{thm}
\label{countingabove}
Let $(a_{-},a_{+},r,R)$ be a quadruple of sequences satisfying Assumptions 
\ref{ass:assume1}-\ref{ass:assume3}, and 
$\nu\in\mathbb{N}$ as in Proposition \ref{proper}.
If there exist $A,a>0$ such that  
$\nu\geq 1+(1-\log A)/a$ and $R(k)/r(k)\geq Ae^{ak}$
for any integer $k\geq\nu$, then the following claims hold:
\begin{enumerate}[label=$(\arabic*)$]
\item \label{item:countingabove}
$N_{\Gamma_{\nu}}(x,R) \leq 2^{2R/a}$
for any $x \in X$ and any $R > 0$;
\item \label{item:disc-spec}
there exists $m_{0}=m_{0}(\Gamma_{\nu})>0$ such that
\[
\mathrm{Spec}_{d}(\square_{X_{\Gamma_{\nu}}})\supset\{4m(m-1)\mid
m\in\mathbb{Z}\text{ and } m> m_{0}\}.
\]
\end{enumerate}
\end{thm}

\begin{ex}
\label{ex-countingabove}
The quadruple $(a_{-},a_{+},r,R)$ in Table \ref{table-quadruple} (1) apply to Theorem \ref{countingabove}.
\end{ex}

Postponing the proof of Theorem \ref{countingabove}, 
we give proofs of Theorems \ref{countingabove-intro} and \ref{non-sharp-spectrum-intro}.
\begin{proof}[Proof of Theorems \ref{countingabove-intro} 
and \ref{non-sharp-spectrum-intro}]
The discontinuous group $\Gamma_{\nu}(a_{-},a_{+},r,R)$
associated to Example \ref{ex-countingabove} 
is non-sharp by Proposition \ref{non-sharp}.
Applying \ref{item:countingabove} and \ref{item:disc-spec} of Theorem \ref{countingabove}, 
we get Theorems \ref{countingabove-intro} and \ref{non-sharp-spectrum-intro},
respectively.
\end{proof}

To prove Theorem \ref{countingabove} \ref{item:disc-spec}, 
namely, to construct an infinite subset of the discrete spectrum, 
we use the following fact, established by Kassel-Kobayashi \cite{KaKob16},
and Theorem \ref{countingabove} \ref{item:countingabove} imply Theorem \ref{countingabove} \ref{item:countingabove} immediately:

\begin{fact}
[{\cite[Thm.\ 3.8 (1) and its proof]{KaKob16}}]
\label{KaKo}
Let $\Gamma$ be a discontinuous group for $X=\mathrm{AdS}^{3}$
satisfying the exponential growth condition (\ref{count}).
Then there exists $m_{0}(\Gamma)>0$ such that 
\[
\mathrm{Spec}_{d}(\square_{X_{\Gamma}})\supset\{4m(m-1)\mid m\in\mathbb{Z}
\text{ and } m> m_{0}(\Gamma)\}.
\] 
\end{fact}

\begin{rem}
Kassel-Kobayashi constructed in \cite[Cor.\ 9.10]{KaKob16}
an infinite subset of the discrete spectrum $\mathrm{Spec}_{d}(\square_{X_{\Gamma}})$
which is stable under small deformations of a sharp discontinuous group $\Gamma$ in $G$. Then does there exist a stable discrete spectrum under small deformations of $\Gamma_{\nu}(a_{-},a_{+},r,R)$?
Since $\Gamma_{\nu}(a_{-},a_{+},r,R)$ is infinitely generated, 
we need to treat deformation carefully.

A small deformation of $\Gamma_{\nu}(a_{-},a_{+},r,R)$ in $G$ is 
neither injective nor discrete in general.
Therefore, there exists no stable discrete spectrum in the sense of 
Kassel-Kobayashi \cite{KaKob16}.
However, the following question looks reasonable:
does there exist an infinite subset of the discrete spectrum which 
is stable under small deformations of $\Gamma_{\nu}(a_{-},a_{+},r,R)$
which act on $\mathrm{AdS}^{3}$ properly discontinuously?

\end{rem}

We prepare some results needed for the proof of Theorem \ref{countingabove} \ref{item:countingabove}.
The following lemma was proved in \cite{KaKob16} in the general setting where
$X$ is a reductive symmetric space.
Since it plays a crucial role in proving Theorem \ref{countingabove},
we give an elementary proof for $X=\mathrm{AdS}^{3}$
for the convenience of the reader.

\begin{lem}
[{\cite[Lem.\ 4.4 and 4.17]{KaKob16}}]
\label{ineqofnu}
For $(\alpha,\beta) \in G$ and $x \in X=\mathrm{AdS}^{3}$, 
\[\|(\alpha,\beta)x\| + \|x\| \geq \left| \|\alpha\|-\|\beta\| \right|.\]
\end{lem}

\begin{proof}
By Lemma \ref{prop-pdist}, we have
\begin{align}
\label{ineqofnu-step1}
\|(\alpha,\beta)x\| &=\|\alpha x\beta^{-1}\|\geq\left| \|\alpha x\| - \|\beta\| \right|, \\
\label{ineqofnu-step2}
\|x\|&=\|\alpha^{-1}\alpha x\|\geq |\|\alpha\|-\|\alpha x\||.
\end{align}
Summing up (\ref{ineqofnu-step1}) and (\ref{ineqofnu-step2}),
we have 
$\|(\alpha,\beta)x\| + \|x\|\geq \left|\|\alpha\|-\|\beta\|\right|$.
\end{proof}

\begin{fact}
[{\cite[Def--Lem.\ 4.20]{KaKob16}}]
\label{Dirichlet}

Let $\Gamma$ be a discontinuous group for $X=\mathrm{AdS}^{3}$. Then, 
the set
\[\mathcal{D}_{X_{\Gamma}} := \{x \in X
\mid \forall \gamma \in \Gamma,\ \|\gamma x\| \geq \|x\| \}.\]
is a fundamental domain of $X$ for the action of $\Gamma$.
In particular, $\Gamma \mathcal{D}_{X_{\Gamma}} = X$.
\end{fact}

Therefore, we may assume $x\in\mathcal{D}_{X_{\Gamma}}$ to study $N_{\Gamma}(x,R)$.

\begin{lem}
[cf. {\cite[Lem.\ 4.21]{KaKob16}}]
\label{uniformineq}
For any $x \in \mathcal{D}_{X_{\Gamma}}$ 
and any $\gamma = (\alpha,\beta) \in \Gamma$, 
\[\|\gamma x\| \geq \frac{1}{2}\left|\|\alpha\| - \|\beta\| \right|.\]
\end{lem}

\begin{proof}
Let $\gamma=(\alpha,\beta) \in \Gamma$ and $x \in \mathcal{D}_{X_{\Gamma}}$.
Then we have 
\[2\|\gamma x\| \geq \|\gamma x\| + \|x\|
\geq \left|\|\alpha\|-\|\beta\| \right|\]
by the definition of $\mathcal{D}_{X_{\Gamma}}$ and Lemma \ref{ineqofnu},
and thus Lemma \ref{uniformineq} holds.
\end{proof}

\begin{lem}
\label{sum}
We set
$\displaystyle
S(R):=\bigcup_{m = 1}^{\infty} \left\{ (k_1,\ldots,k_m) \in \mathbb{N}_{+}^{m} \middle|
\sum_{i = 1}^{m} k_i \leq R \right\}$ for any $R\in\mathbb{N}$.
Then the cardinality of $S(R)$ equals $2^{R}-1$.
\end{lem}

\begin{proof}
For $(k_{1},\ldots,k_{m})\in S(R)$,
we define the binary number 
$1\underbrace{0\ldots0}_{k_1-1}\ldots1\underbrace{0\ldots0}_{k_m-1}$, 
which induces a bijection $S(R)\xrightarrow{\sim} \mathbb{Z}\cap [1,2^{R}-1]$.
Hence $\# S(R)=2^{R}-1$.
\end{proof}

We are ready to prove Theorem \ref{countingabove}.
\begin{proof}[Proof of Theorem \ref{countingabove}]
We now prove \ref{item:countingabove}. 
Let $\Gamma_{\nu}\equiv\Gamma_{\nu}(a_{-},a_{+},r,R)$.
Take $\nu\in\mathbb{N}$ as in Proposition \ref{proper} and assume $\nu\geq 1+(1-\log A)/a$.
Then recall $\varepsilon(\nu)(<1)\leq 2$.
Moreover, take $\gamma\in F^{\infty}_{\nu}\smallsetminus\{e\}$.
Let $m:=\ell(\gamma)$ be the word length of $\gamma$, and 
we write $\gamma=\gamma_{k_1}^{s_1}\cdots \gamma_{k_m}^{s_m}$ for 
the reduced expression where $s_1,\ldots,s_{m} \in\{1,-1\}$ and 
$k_1,\ldots,k_{m} \geq \nu(\geq 2)$.
By Proposition \ref{prop:key-ineq} \ref{prop:diffofmu},
\begin{align}
\label{diffofmu-example}
\frac{1}{2}(\| j(\gamma) \| - \| \rho(\gamma) \|)
&\geq\sum_{i=1}^{m}\left(\log\left(\frac{R(k_{i})}{r(k_{i})}\right)-\frac{1}{2}\varepsilon(\nu)\right)
\nonumber \\
&\geq
\sum_{i=1}^{m} (\log A + ak_i-1) \nonumber \\
&=a\sum_{i=1}^{m}\left( k_{i} - \frac{1-\log A}{a}\right) \nonumber \\
&\geq a\sum_{i=1}^{m}\left( k_{i} - (\nu-1)\right).
\end{align}

To prove \ref{item:countingabove}, we may and do assume 
$x \in \mathcal{D}_{X_{\Gamma_{\nu}}}$ by Fact \ref{Dirichlet}.
Suppose $(j(\gamma),\rho(\gamma))x \in B(R)$.
Then $\left|\|j(\gamma)\| - \|\rho(\gamma)\| \right|/2\leq R$ 
by Lemma \ref{uniformineq}.
Hence $(k_1-(\nu-1),\ldots,k_m-(\nu-1))\in S(\lfloor R/a\rfloor)$ by the inequality (\ref{diffofmu-example}),
where $\lfloor R\rfloor$ is the largest integer less than or equal to $R$ for $R\in\mathbb{R}$, and in particular, $m\leq R/a$.
The number of such $(k_1,\ldots,k_m)$
is at most $2^{R/a} - 1$ by Lemma~\ref{sum}
and thus the number of such
$\gamma = \gamma_{k_1}^{s_1}\cdots \gamma_{k_m}^{s_m}(\neq e)$ with 
$s_1,\ldots,s_{m} \in\{1,-1\}$ is at most $(2^{R/a}-1)2^{R/a}$.
Hence we obtain $N_{\Gamma_{\nu}}(x,R) \leq (2^{R/a} - 1)2^{R/a} + 1 \leq 2^{2R/a}$,
which proves \ref{item:countingabove}. The assertion \ref{item:disc-spec} follows from \ref{item:countingabove} and Fact \ref{KaKo}.
\end{proof}

\appendix
\section{Proof of the formula (\ref{riemann})}

The formula (\ref{riemann}) is probably well known, but 
we give a proof for the reader's convenience.
It suffices to show the following proposition:
\begin{prop}
Let $\Gamma$ be a discrete group of isometries of a complete Riemannian manifold $X$.
We write $B(x; R)$ for the ball of radius $R$ centered at $x\in X$. 
Fix a point $x_{0}\in X$. Then,
\[
\forall x\in X,\ \exists c>0,\ \#(\Gamma x\cap B(x_{0}; R))\leq
\frac{\mathrm{vol}(B(x_{0}; R+c))}{\mathrm{vol}(B(x; c))}.
\]
\end{prop}

\begin{proof}
Let $x\in X$. We denote by $d_{X}(\cdot,\cdot)$ the distance of $X$,
and set $3c:=\inf \{d_{X}(x,y)\mid y\in \Gamma x\smallsetminus\{x\}\}$. 
By \cite[Ch.\ IV, Thm.\ 2.2]{Helgason_symmetric},
the orbit $\Gamma x$ of the discrete group $\Gamma$ is discrete, hence
we have $c>0$. By the triangle inequality, the inclusion relation
\begin{align}
\label{ineq:inclusion_riemann}
\bigcup_{y\in \Gamma x \cap B(x_{0}; R)} B(y; c) \subset B(x_{0}; R+c)
\end{align}
holds. We note that the left hand side of (\ref{ineq:inclusion_riemann}) is a disjoint union by the definition of $c$. Since $\mathrm{vol}(B(y; c))=\mathrm{vol}(B(x; c))$ for any $y\in \Gamma x$, we obtain the desired inequality taking the volumes of the both hand sides of (\ref{ineq:inclusion_riemann}).
\end{proof}

\section*{Acknowledgements}
The author would like to express his sincere gratitude 
to Professor Toshiyuki Kobayashi for his support and encouragement. 
He would also like to show his appreciation to Dr.\ Yosuke Morita for his helpful comments. 
Thanks are also due to an anonymous referee 
for comments and suggestions that improved this paper.
This work was supported  by JSPS KAKENHI Grant Number 18J20157 
and the Program for Leading Graduate Schools, MEXT, Japan.


\begin{thebibliography}{10}

\bibitem{Ben96}
Y.~Benoist.
\newblock Actions propres sur les espaces homog\`enes r\'{e}ductifs.
\newblock {\em Ann. of Math. (2)}, 144(2):315--347, 1996.

\bibitem{EsMc93}
A.~Eskin and C.~McMullen.
\newblock Mixing, counting, and equidistribution in {L}ie groups.
\newblock {\em Duke Math. J.}, 71(1):181--209, 1993.

\bibitem{GuKa17}
F.~Gu\'{e}ritaud and F.~Kassel.
\newblock Maximally stretched laminations on geometrically finite hyperbolic
  manifolds.
\newblock {\em Geom. Topol.}, 21(2):693--840, 2017.

\bibitem{Helgason_symmetric}
S.~Helgason.
\newblock {\em Differential geometry and symmetric spaces}.
\newblock Pure and Applied Mathematics, Vol. XII. Academic Press, New
  York-London, 1962.

\bibitem{Ka09}
F.~Kassel.
\newblock {\em \emph{{Q}uotients compacts d'espaces homog{\`e}nes r{\'e}els ou
  p-adiques}}.
\newblock PhD thesis, Universit{\'e} Paris-Sud, 2009.

\bibitem{KaKob16}
F.~Kassel and T.~Kobayashi.
\newblock Poincar\'{e} series for non-{R}iemannian locally symmetric spaces.
\newblock {\em Adv. Math.}, 287:123--236, 2016.

\bibitem{KaKob19}
F.~Kassel and T.~Kobayashi.
\newblock Spectral analysis on standard locally homogeneous spaces.
\newblock arXiv:1912.12601, preprint.

\bibitem{Kob89}
T.~Kobayashi.
\newblock Proper action on a homogeneous space of reductive type.
\newblock {\em Math. Ann.}, 285(2):249--263, 1989.

\bibitem{Kob96}
T.~Kobayashi.
\newblock Criterion for proper actions on homogeneous spaces of reductive
  groups.
\newblock {\em J. Lie Theory}, 6(2):147--163, 1996.

\bibitem{Kob98}
T.~Kobayashi.
\newblock Deformation of compact {C}lifford-{K}lein forms of
  indefinite-{R}iemannian homogeneous manifolds.
\newblock {\em Math. Ann.}, 310(3):395--409, 1998.

\bibitem{Ko01}
T.~Kobayashi.
\newblock Discontinuous groups for non-{R}iemannian homogeneous spaces.
\newblock In {\em Mathematics unlimited---2001 and beyond}, pages 723--747.
  Springer, Berlin, 2001.

\bibitem{Mi68}
J.~Milnor.
\newblock A note on curvature and fundamental group.
\newblock {\em J. Differential Geometry}, 2:1--7, 1968.

\end{thebibliography}
\end{document}